\documentclass[12pt]{article}
\usepackage[english]{babel}
\usepackage{upref,amsfonts,amsxtra,latexsym,color,mathtools}
\usepackage{amssymb,amsmath,amsthm,a4wide,epsf,mathrsfs,verbatim,hyperref}
\usepackage[all]{xy}

\newcommand{\C}{{\mathbb C}}
\newcommand{\Z}{{\mathbb Z}}
\newcommand{\N}{{\mathbb N}}

\newcommand{\R}{{\mathbb R}}

\newcommand{\ep}{{\varepsilon}}
\newcommand{\cF}{{\mathcal F}}
\newcommand{\cB}{{\mathcal B}}
\newcommand{\cA}{{\mathcal A}}

\newcommand{\cI}{{\mathcal I}}
\newcommand{\cK}{{\mathcal K}}

\newcommand{\cP}{{\mathrm{Ped}}}

\newcommand{\cU}{{\mathcal U}}
\newcommand{\cO}{{\mathcal O}}
\newcommand{\cJ}{{\mathcal J}}

\newcommand{\Cs}{{$C^*$-al\-ge\-bra}}
\newcommand{\sh}{{$^*$-ho\-mo\-mor\-phism}}

\newtheorem{theorem}{Theorem}[subsection]
\newtheorem{lemma}[theorem]{Lemma}
\newtheorem{corollary}[theorem]{Corollary}
\newtheorem{proposition}[theorem]{Proposition}

\theoremstyle{definition}
\newtheorem{definition}[theorem]{Definition}
\newtheorem{remark}[theorem]{Remark}
\newtheorem{example}[theorem]{Example}

\numberwithin{equation}{section}
\numberwithin{theorem}{section}

\date{\empty}
\title{Fixed-points in the cone of traces on a \Cs}

\author{Mikael R\o rdam\thanks{Supported  by the Danish Council for Independent Research, Natural Sciences, and the Danish National Research Foundation (DNRF) through the Centre for Symmetry and Deformation at the University of Copenhagen. }}

\begin{document} 

\maketitle

\centerline{\emph{Dedicated to the memory of John Roe}}

\begin{abstract}

\noindent Nicolas Monod introduced in \cite{Monod:cones} the class of groups with the fixed-point property for cones, characterized by always admitting a non-zero fixed point whenever acting (suitably) on proper weakly complete cones. He proved that his class of groups contains the class of groups with subexponential growth and is contained in the class of supramenable groups. In this paper we investigate what Monod's results say about the existence of invariant traces on (typically non-unital) \Cs s equipped with an action of a group with the fixed-point property for cones. As an application of these results we provide results on the existence (and non-existence) of traces on the (non-uniform) Roe algebra. 
\end{abstract}

\section{Introduction} 

\noindent Whenever a discrete amenable group acts on a unital \Cs{} with at least one tracial state, then the \Cs{} admits an invariant tracial state, and the crossed product \Cs{}  admits a tracial state. This, moreover, characterizes amenable groups. The purpose of this paper is to find statements, similar to this well-known fact, about the existence of invariant traces on non-unital \Cs s using the results of the recent paper by Monod, \cite{Monod:cones}, in which the class of groups with the \emph{fixed-point property for cones} is introduced and developed.

In \cite{Monod:cones}, a group is said to have the fixed-point property for cones if whenever it acts continuously on a proper weakly complete cone (embedded into a locally convex topological vector space), such that the action is of \emph{cobounded type} and \emph{locally bounded}, then there is a non-zero fixed point in the cone. Being of cobounded type is an analog of an action on a locally compact Hausdorff space being \emph{co-compact}, see \cite[Definition 2]{Monod:cones}. The action is locally bounded if there is a non-zero bounded orbit, see \cite[Definition 1]{Monod:cones}. 

Even the group of integers, $\Z$, can fail to leave invariant any non-zero trace when acting on a non-unital \Cs. For example, the stabilization of the Cuntz algebra $\cO_2$ is the crossed product of the stabilized CAR-algebra $\cA$ with an action of $\Z$ that scales the traces on $\cA$ by a factor of 2. In particular, there are no non-zero invariant traces on $\cA$. The cone of (densely defined lower semi-continuous) traces on $\cA$ is isomorphic to the cone $[0, \infty)$, and the induced action of $\Z$ on this cone is multiplication by 2, which of course fails to be locally bounded. 

The action of any group $G$ on a locally compact Hausdorff space $X$ induces a locally bounded action on the cone of Radon measures on $X$ (which again is the same as the cone of densely defined lower semi-continuous traces on $C_0(X)$). However, any infinite group can act on the locally compact non-compact Cantor set $\mathbf{K}^*$ in a non co-compact way so that there are no non-zero invariant Radon measures, cf.\ \cite[Section 4]{MatuiRor:universal}. Such an action of $G$ on the Radon measures on $\mathbf{K}^*$ fails to be of cobounded type. 

The two examples above explain why one must impose conditions on the action, such as being of cobounded type and being locally bounded, to get meaningful results on when invariant traces exist. 

Monod proves in \cite{Monod:cones} that the class of groups with the fixed-point property for cones contains the class of groups of subexponetial growth and is contained in the class of supramenable groups introduced by Rosenblatt in \cite{Rosenblatt:supramenable}. It is not known if there are supramenable groups of exponential growth, so the three classes of groups could coincide, although the common belief seems to be that they all are different. Monod proved a number of permanence properties for his class of groups, prominently including that it is closed under central extensions, see \cite[Theorem 8]{Monod:cones}. He also shows that the property of having the fixed-point property for cones can be recast in several ways, including the property that for each non-zero positive function $f \in \ell^\infty(G)$ there is a non-zero  invariant  positive linear functional (called an \emph{invariant integral} in \cite{Monod:cones}) on the subspace $\ell^\infty(G,f)$ of all bounded functions $G$-dominated by $f$.

We begin our paper in Section~\ref{sec:traces-ideals} by recalling properties of possibly unbounded positive traces on (typically non-unital) \Cs s, including when they are lower semi-continuous and when they are singular. By default, all traces in this paper are assumed to be positive. Interestingly, many of the traces predicted by Monod turn out to be singular. Traces on \Cs s were treated systematically already by Dixmier in \cite{Dixmier:traces}.  Elliott, Robert and Santiago consider in  \cite{EllRobSan:traces} the cone of lower semi-continuous traces as an invariant of the \Cs, they derive useful properties of this cone, and they make the point that such traces most conveniently are viewed as maps defined on the positive cone of the \Cs{} taking values in $[0,\infty]$. While this indeed is a convenient way to describe unbounded traces, and one we  in part shall use here, the cone structure from the point of view of this paper is sometimes better portrayed when traces are viewed as linear functionals on a suitable domain: a hereditary symmetric algebraic ideal in the \Cs.

By a theorem of G.K.\ Pedersen one can identify the cone of  lower semi-continuous densely defined traces on a \Cs{} with the cone of traces defined on its Pedersen ideal. In the case where the primitive ideal space of the \Cs{} is compact,  we show that there are non-zero lower semi-continuous densely defined traces if and only if the stabilization of the \Cs{} contains no full properly infinite projections, thus extending well-known results from both the unital and the simple case. A similar compactness condition appears in our reformulation of coboundedness of the action of the group on the cone of traces. 

In Section~\ref{sec:inv-traces} we explain when a \Cs{} equipped with an action of a group $G$ with the fixed-point property for cones admits an invariant densely defined lower semi-continuous trace. Given that the cone of densely defined lower semi-continuous traces is always  proper and weakly complete, all
we have to do is to explain when the action of the group on this cone is of cobounded type, respectively, when it is locally bounded. The former can be translated into a compactness statement, as mentioned above, while the latter just means that there exists a non-zero trace which is bounded on all $G$-orbits. 

In Section~\ref{sec:inv-integrals} we examine the situation where the group $G$ acts on $\ell^\infty(G)$ with the aim of describing for which positive $f \in \ell^\infty(G)$ the invariant integrals on $\ell^\infty(G)$ normalizing $f$ are lower semi-continuous, respectively, singular. As it turns out, frequently they must be singular. We use this to give an example of a $G$-invariant densely defined trace on a \Cs{} that does not extend to a trace on the crossed product.

Finally, in Section~\ref{sec:roe}, we consider the particular example of invariant traces on $\ell^\infty(G,\cK)$ and traces on the Roe algebra $\ell^\infty(G,\cK) \rtimes G$. We show that $\ell^\infty(G,\cK)$ only has the ``obvious'' densely defined lower semi-continuous traces, and hence that it never has non-zero $G$-invariant ones, when $G$ is infinite. In many cases, however, there are invariant lower semi-continuous traces with smaller domains, such as domains defined by projections in $\ell^\infty(G,\cK)$. Specifically we show that any projection in  $\ell^\infty(G,\cK)$ whose dimension (as a function on $G$) grows subexponentially is normalized by an invaraint  lower semi-continuous trace if $G$ has the fixed-point property for cones. In general, for any non-locally finite group $G$, there are (necessarily exponentially growing) projections in $\ell^\infty(G,\cK)$ not normalized by any invariant trace, and which are properly infinite in the Roe algebra, while the Roe algebra of a locally finite group is always stably finite.

I thank Nicolas Monod, Nigel Higson, Guoliang Yu, and Eduardo Scarparo for useful discussions related to this paper. I also thank the referee for suggesting improvements of the exposition and for pointing out a couple mistakes in an earlier version of this paper.

\section{Hereditary ideals and cones of traces}
\label{sec:traces-ideals}

\noindent By a \emph{hereditary ideal} $\cI$ in a \Cs{} $\cA$ we shall mean an algebraic two-sided self-adjoint ideal satisfying the hereditary property: if $0 \le a \le b$, $b \in \cI$ and $a \in \cA$, then $a \in \cI$. If $x^*x \in \cI$ whenever $xx^* \in \cI$, for all $x \in \cA$, then we say that $\cI$ is \emph{symmetric}. All closed two-sided ideals are hereditary and symmetric. 

If a group $G$ acts on the \Cs{} $\cA$, then we refer to an ideal $\cI$ as being \emph{$G$-invariant} (or just invariant) if it is invariant under the group action. If $M$ is a subset of $\cA$, then let $\cI_\cA(M)$, respectively, $\overline{\cI_\cA}(M)$, denote the smallest hereditary ideal in $\cA$, respectively, the smallest closed two-ideal in $\cA$,  which contains $M$; and let similary $\cI_\cA^G(M)$ and $\overline{\cI_\cA^G}(M)$ denote the smallest $G$-invariant hereditary ideal, respectively,  the smallest $G$-invariant closed two-ideal in $\cA$ which contains $M$.

\begin{example} \label{ex:Pedersen}
 The \emph{Pedersen ideal},  $\cP(\cA)$, of a \Cs{} $\cA$ is the (unique) smallest dense ideal in $\cA$, see \cite[Section 5.6]{Ped:C*-aut}. It is a hereditary symmetric ideal, and, even better: for each $x \in \cP(\cA)$, the  hereditary sub-\Cs, $\overline{x\cA x^*}$, of $\cA$ is contained in  $\cP(\cA)$. In particular, the Pedersen ideal of $\cA$ is closed under continuous function calculus on its normal elements, as long as the continuous function vanishes at $0$. If $x\in \cA$ is such that $x^*x \in \cP(\cA)$, then $x \in \cP(\cA)$, which shows that the Pedersen ideal is also symmetric. 
\end{example}

\begin{example} Not all self-adjoint two-sided ideals in a \Cs{} are hereditary. Consider for example the commutative \Cs{} $\cA = C([-1,1])$, the element $f \in C([-1,1])$ given by $f(t) = |t|$, for $t \in [-1,1]$, and the (two-sided) self-adjoint ideal $\cI = \cA f$ in $\cA$. The function $g(t) = \max\{-\frac12 t,t\}$, $t \in [-1,1]$, then satisfies $0 \le g \le f$, but $g \notin \cI$.

Hereditary ideals need not be spanned (or even generated) by their positive elements.  Indeed, take again $\cA = C([-1,1])$ and let $\cI = \cA \, \iota$, where $\iota(t) = t$, for $t \in [-1,1]$. If $f \in \cI \cap \cA^+$, then $f = g \cdot \iota$, for some $g \in \mathcal{P}$, where $\mathcal{P}$ is the set of functions $g \in \cA$ such that $g(t) \le 0$, for $t \in [-1,0]$, and $g(t) \ge 0$, for $t \in [0,1]$. Since $g(0)=0$, for all $g \in \mathcal{P}$, it is not possible to write $\iota \in \cI$ as a linear combination of functions in $\cI \cap \cA^+$. To see that $\cI$ is hereditary, suppose that $0 \le h \le f$, where $f \in \cI \cap \cA^+$ and $h \in \cA$. Then $f = g \cdot \iota$, for some $g \in \mathcal{P}$. Hence $|h(t)/t| \le |g(t)|$, for all $t \ne 0$, from which we see that $h \in \cI$. 

Let us also note that (algebraic) two-sided ideals need not be self-adjoint. Consider  the commutative \Cs{} $\cA = C(\mathbb{D})$, where $\mathbb{D}$ is the closed unit disk in the complex plane, the function $f(z) = z$, $z \in \mathbb{D}$, and the (two-sided) ideal $\cI = \cA f$ in $\cA$. Then $f$ belongs to $\cI$, but $f^* = \bar{f}$ does not.

\end{example}

\noindent For a \Cs{} $\cA$ let $\widetilde{\cA}$ denote the unitization of  $\cA$, when it is non-unital, or $\cA$ itself when it already is unital.

\begin{lemma} \label{lm:cone}
Let $\cA$ be a \Cs, and let $C$ be a subcone of $\cA^+$ satisfying:
\begin{enumerate}
\item If $a \in C$ and $x \in \widetilde{\cA}$, then $x^*ax \in C$,
\item $C$ is hereditary: if $0 \le a \le b$, $a \in \cA$ and $b \in C$, then $a \in C$.
\end{enumerate}
Let $\cI$ be the linear span of $C$. Then $\cI$ is a hereditary ideal in $\cA$ and  $\cI \cap \cA^+ = C$. If $C$ is symmetric (in the sense that $x^*x \in C$ implies $xx^* \in C$, for all $x \in \cA$), then so is $\cI$.

Conversely, if $\cI$ is a hereditary ideal in $\cA$ and if $C = \cI \cap \cA^+$, then $C$ is a subcone of $\cA^+$ satisfying (i) and (ii) above. The span, $\cI_0$, of $C$ is a hereditary subideal of $\cI$; and $\cI_0 = \cI$ if and only if $\cI$ is generated as a hereditary ideal by its positive elements. 
\end{lemma}

\begin{proof} It follows from (i) that if $x \in \cA$ and $a \in C$, then
\begin{eqnarray*}
xa+ax^* &=& (x+1)a(x+1)^* - xax^* - a,\\
i(xa-ax^*) &=& (x-i)a(x-i)^* - xax^*-a,
\end{eqnarray*}
belong to $\cI$, whence $xa$ and $ax^*$ belong to $\cI$. This shows that $\cI$ is an ideal in $\cA$.  Clearly, $\cI$ is self-adjoint. 

A subcone $C$ of $\cA^+$ satisfies $\mathrm{span}(C) \cap \cA^+ = C$ if and only if whenever $a,b \in  C$ are such that $a \le b$, then $b-a \in C$. Hereditary cones clearly have this property, so $\cI \cap \cA^+ = C$. This also shows that $\cI$ is hereditary (because $C$ is hereditary).

It is clear that $C = \cI \cap \cA^+$ has the stated properties if $\cI$ is a hereditary ideal of $\cA$; and $\cI_0$ is a hereditary ideal of $\cA$ by the first part of the lemma. It is contained in $\cI$ and contains by definition all positive elements of $\cI$. That proves the last claim of the lemma.
\end{proof}

\begin{corollary} \label{cor:cone}
Let $\cA$ be a \Cs{} and let $M$ be a non-empty subset of $\cA^+$. Let $C$ be the set of all elements $a \in \cA^+$ for which there exist $n \ge 1$, $e_1, e_2, \dots, e_n \in M$, and $x_1,x_2, \dots, x_n \in \widetilde{\cA}$ such that $a \le \sum_{j=1}^n x_j^*e_jx_j$. Then:
\begin{enumerate}
\item $C$ is a subcone of $\cA^+$ satisfying (i) and (ii) of Lemma~\ref{lm:cone}; 
\item $\cI_\cA(M) = \mathrm{span}(C)$;
\item $\cI_\cA(M) \cap \cA^+ = C$. 
\end{enumerate}
\end{corollary}

\begin{proof} It is clear that (i) holds, so Lemma~\ref{lm:cone} implies that $\cI:=\mathrm{span}(C)$ is a hereditary ideal in $\cA$ satisfying $\cI \cap \cA^+ = C$. As $M \subseteq C \subseteq \cI$ we conclude that $\cI_\cA(M) \subseteq \cI$. Conversely, $C \subseteq \cI_\cA(M)$, so $\cI \subseteq \cI_\cA(M)$. 
\end{proof} 

\begin{corollary} \label{lm:ideal}
Let $\cA$ be a \Cs{}, let $\alpha$ be an action of a  group $G$ on $\cA$, and let $M$ be a non-empty subset of $\cA^+$.  Then
\begin{enumerate}
\item  $\cI_\cA^G(M) = \cI_\cA(G.M)$, where $G.M = \{\alpha_t(e) : t\in G, e \in M\}$.
\item An element $a \in \cA^+$ belongs to $\cI_\cA^G(M)$ if and only if there exist $n \ge 1$, $t_1, t_2, \dots, t_n \in G$, $y_1,y_2, \dots, y_n \in \widetilde{\cA}$, and $e_1,e_2, \dots e_n \in M$ such that
$a \le \sum_{j=1}^n y_j^*\alpha_{t_j}(e_j)y_j$.
\end{enumerate}
\end{corollary}

\begin{proof} (i). As $G.M \subseteq \cI_\cA^G(M)$ we see that $\cI_\cA(G.M) \subseteq \cI_\cA^G(M)$. Conversely, as $\cI_\cA(G.M)$ is $G$-invariant, it contains $\cI_\cA^G(M)$. Part (ii) follows from (i) and from Corollary~\ref{cor:cone}.
\end{proof}

\noindent For each non-empty subset $M$ of $\cA^+$ denote by $\cJ_\cA(M)$ the smallest \emph{symmetric} hereditary ideal in $\cA$ containing $M$. If a group $G$ acts on $\cA$, then denote by $\cJ_\cA^G(M)$ the smallest symmetric hereditary $G$-invariant ideal in $\cA$ containing $M$. Since closed two-sided ideals in a \Cs{} always are symmetric, we have
\begin{equation} \label{eq:J-I}
\cI_\cA(M) \subseteq \cJ_\cA(M) \subseteq \overline{\cI_\cA}(M), \qquad \cI_\cA^G(M) \subseteq \cJ_\cA^G(M) \subseteq \overline{\cI_\cA^G}(M).
\end{equation}

\begin{lemma} \label{lm:symmetric-proj}
Let $\cA$ be a \Cs{} and let $M$ be a non-empty subset of projections in $\cA$, then $\cI_\cA(M) = \cJ_\cA(M) = \cP(\cA_0)$, where $\cA_0 = \overline{\cI_\cA}(M)$. 
If $\cA$ is equipped with an action of a group $G$, then $\cI_\cA^G(M) = \cJ_\cA^G(M) = \cP(\cA_1)$, where $\cA_1 = \overline{\cI_\cA^G}(M)$.
\end{lemma}

\begin{proof} We have $M \subseteq \cP(\cA_0) \subseteq \cI_\cA(M) \subseteq \cJ_\cA(M) \subseteq \cA_0$ (the first inclusion holds because each projection in a \Cs{} belong to its Pedersen ideal, and  second inclusion holds because $\cI_\cA(M)$ is a dense ideal in  $\cA_0$). As $\cP(\cA_0)$ is a hereditary symmetric ideal, which contains $M$, cf.\ Example~\ref{ex:Pedersen}, $\cJ_\cA(M) \subseteq \cP(\cA_0)$, so $\cP(\cA_0) = \cI_\cA(M) = \cJ_\cA(M)$. The second part of the statement follows from the first part applied to $G.M$ (instead of $M$).
\end{proof}

\begin{lemma} \label{lm:J-positive} 
Let $\cA$ be a \Cs{} and let $M$ be a non-empty set of positive elements in $\cA$. Then $\cJ_\cA(M)$ is the linear span of its positive elements. If $\cA$ is equipped with an action of a group $G$, then the same holds for $\cJ_\cA^G(M)$.
\end{lemma}

\begin{proof} Set $C = \cJ_\cA^G(M) \cap \cA^+$ and let $\cJ_0$ be the linear span of $C$. Then $\cJ_0$ is a hereditary ideal in $\cA$ by Lemma~\ref{lm:cone}. Since $\cJ_\cA^G(M)$ is symmetric and $G$-invariant, the same holds for $C$, and hence for $\cJ_0$. Moreover, $M \subseteq C \subseteq \cJ_0 \subseteq \cJ_\cA^G(M)$. As $\cJ_\cA^G(M)$ is the smallest ideal with these properties, $\cJ_0 = \cJ_\cA(M)$.
The first claim is proved in a similar manner.
\end{proof}

\begin{definition} \label{def:trace}
Let $\cA$ be a \Cs. Denote by $T^+(\cA)$ the cone of traces on the positive cone of $\cA$, i.e., the set of additive homogeneous maps $\tau \colon \cA^+ \to [0,\infty]$ satisfying $\tau(x^*x) = \tau(xx^*)$, for all $x \in \cA$.

For each hereditary symmetric ideal $\cI$ in $\cA$, let $T(\cI,\cA)$ denote the cone of  linear traces on $\cI$, i.e., the set of positive linear maps $\tau \colon \cI \to \C$ satisfying $\tau(x^*x) = \tau(xx^*)$, whenever $x\in\cA$ is such that $x^*x$ (and hence $xx^*$) belong to $\cI$. We refer to $\cI$ as the \emph{domain} of $\tau$. If the domain of $\tau$ is a dense ideal of $\cA$ (in which case it will contain the Pedersen ideal of $\cA$), then $\tau$ is said to be a densely defined trace on $\cA$. 
\end{definition}

\noindent The cone of traces, here denoted by $T^+(\cA)$, is in \cite{EllRobSan:traces} denoted by $T(\cA)$.
\emph{Note that all traces by default are assumed to be positive.} The following easy fact will be used repeatedly:

\begin{lemma} \label{lm:traceinequality} Let $\tau$ be a trace on a \Cs{} $\cA$. Then $\tau(x^*ax) \le \|x\|^2 \tau(a)$, for all $a \in \cA^+$ (in the domain of $\tau$) and all $x$ in the unitization of $\cA$.
\end{lemma}

\noindent
The ``$\ep$-cut-down'' $(a-\ep)_+$ of $a \in \cA^+$ appearing in  part (i) of the lemma below is defined by applying the continuous positive function $t \mapsto \max\{t-\ep,0\}$ to $a$. 

A trace in $T^+(\cA)$, or a linear trace on $\cA$ defined on a given domain, is said to be \emph{lower semi-continuous} if one of the equivalent conditions in the following lemma holds.

\begin{lemma} \label{lm:lsc-eq} Let $\cA$ be a \Cs. The following conditions are equivalent for each trace $\tau$  in $T^+(\cA)$ (or  for each  linear trace on $\cA$):
\begin{enumerate}
\item $\tau(a) = \sup_{\ep>0} \tau((a-\ep)_+)$ for each $a \in \cA^+$ (in the domain of $\tau$),
\item whenever $\{a_n\}_{n=1}^\infty$ is an increasing sequence in $\cA^+$ (in the domain of $\tau$) converging in norm to $a \in \cA^+$ (in the domain of $\tau$), then $\tau(a) = \lim_{n\to\infty} \tau(a_n)$,
\item whenever $\{a_n\}_{n=1}^\infty$ is a sequence in $\cA^+$ (in the domain of $\tau$) converging in norm to $a \in \cA^+$ (in the domain of $\tau$), then $\tau(a) \le  \liminf_{n\to\infty} \tau(a_n)$.
\end{enumerate}
\end{lemma}

\begin{proof} (ii) $\Rightarrow$ (i). To verify (i) one needs only show that $\tau(a) = \lim_{n \to \infty} \tau((a-\ep_n)_+)$ for all sequences $\{\ep_n\}$ decreasing to $0$; but this is  just a special case of (ii). 

(iii) $\Rightarrow$ (ii). If $\{a_n\}_{n=1}^\infty$ is an increasing sequence of positive elements  converging to $a$, then $\tau(a_n) \le \tau(a)$, for all $n$, by positivity of $\tau$. If (iii) holds, then this entails that $\tau(a) = \lim_{n\to\infty} \tau(a_n)$.

(i) $\Rightarrow$ (iii). Let $\ep >0$ be given. Choose $n_0 \ge 1$ such that $\|a_n - a\| <\ep$, for all $n \ge n_0$. It then follows from \cite[Lemma 2.2]{KirRor:pi2} that $(a-\ep)_+ = d_n^*a_nd_n$, for some contractions $d_n$ in $\cA$, for all $n \ge n_0$. Hence $\tau((a-\ep)_+)  \le \tau(a_n)$, by Lemma~\ref{lm:traceinequality}. This shows that $\liminf_{n \to \infty} \tau(a_n) \ge \tau((a-\ep)_+)$. It follows that $\liminf_{n \to \infty} \tau(a_n) \ge \sup_{\ep>0} \tau((a-\ep)_+)$, which proves that (i) implies (iii).
\end{proof}

\begin{theorem}[G.K.\ Pedersen, {\cite[Corollary 3.2]{GKP:Pedersen-ideal}}] 
\label{thm:GKP}
The restriction of any densely defined trace on a \Cs{} $\cA$ to the Pedersen ideal of $\cA$ is automatically lower semi-continuous.
\end{theorem}

\begin{definition} \label{def:lsc-trace}
Denote by $T_{\mathrm{lsc}}(\cA)$ the cone of  linear traces on $\cA$ whose domain is the Pedersen ideal of $\cA$. In other words, $T_{\mathrm{lsc}}(\cA) = T(\cP(\cA),\cA)$.
\end{definition}

\noindent  We can identify $T_{\mathrm{lsc}}(\cA)$ with the set of densely defined lower semi-continuous traces on $\cA$ as follows: Each trace in $T_{\mathrm{lsc}}(\cA)$ is clearly densely defined, and it is lower semi-continuous by Pedersen's theorem. Conversely, if $\tau$ is a lower semi-continuous densely defined trace, then its restriction $\tau_0$  to the Pedersen ideal of $\cA$ belongs to $T_{\mathrm{lsc}}(\cA)$, and $\tau$ is uniquely determined  on its domain by $\tau_0$ by Lemma~\ref{lm:lsc-eq} (i), because $(a-\ep)_+ \in \cP(\cA)$ for all positive $a \in \cA$ and all $\ep >0$.

A trace in $T_{\mathrm{lsc}}(\cA)$ can usually be extended to a lower semi-continuous trace on a larger domain than the Pedersen ideal; and such an extension is unique, see Proposition~\ref{prop:tau->tau'} below and the subsequent discussion.

It follows from Theorem~\ref{thm:GKP} and Lemma~\ref{lm:symmetric-proj} that each linear trace on a \Cs{} $\cA$  with domain $\cJ_\cA(M)$ or $\cJ_\cA^G(M)$ (when $\cA$ has a $G$-action) is lower semi-continuous whenever $M$ is a subset of projections in $\cA$. 

Observe that $T_{\mathrm{lsc}}(B(H)) = \{0\}$, where $B(H)$ is the bounded operators on a separable infinite dimensional Hilbert space $H$, while $T_{\mathrm{lsc}}(\cK(H))$ and the cone of lower semi-continuous traces in $T^+(B(H))$ both are equal to the one-dimensional cone spanned by the canonical trace on $B(H)$, in the latter case viewed as a function $B(H)^+ \to [0,\infty]$. The Dixmier trace is an example of a singular trace on $\cK(H)$. It belongs to $T^+(B(H))$ and to $T(\cI,\cK(H))$, where $\cI \subset \cK(H)$ is its domain, and it is zero on the finite rank operators.

Consider a general (not necessarily densely defined) linear trace $\tau$ on $\cA$ with domain $\cI$.  The closure, $\overline{\cI}$, of $\cI$ is a closed two-sided ideal in $\cA$, and hence, in particular, a \Cs; and $\tau$ is of course densely defined relatively to this \Cs. We have the following inclusions:
$$\cP(\overline{\cI}) \subseteq \cI \subseteq \overline{\cI}.$$
The restriction of $\tau$ to $\cP(\overline{\cI})$ is lower semi-continuous by Theorem~\ref{thm:GKP}. If this restriction is zero, then $\tau$ is said to be \emph{singular}. Each trace $\tau$ on $\cA$ with domain $\cI$ can in a unique way be written as the sum $\tau = \tau_1+\tau_2$  of a lower semi-continuous trace $\tau_1$ and a singular trace $\tau_2$, both with domain $\cI$. The lower semi-continuous part is obtained by restricting $\tau$ to the Pedersen ideal (which is lower semi-continuous) and then extending to a lower semi-continuous trace $\tau_1$ defined on $\cI$ as described in \eqref{eq:tau} and the subsequent comments.

One can smoothly and uniquely pass from a trace in $T^+(\cA)$  to a linear trace defined on its natural (maximal) domain:

\begin{proposition} \label{prop:tau->tau'}
Let $\cA$ be a \Cs, and let $\tau' \in T^+(\cA)$.  Let $C$ be the set of positive elements $a \in \cA$ with $\tau'(a) < \infty$, and  let $\cI$ be the linear span of $C$. Then $\cI$ is a hereditary symmetric ideal in $\cA$,  $\cI \cap \cA^+ = C$, and there is a unique linear trace $\tau$ with domain $\cI$ that agrees with $\tau'$ on $C$.

We can recover $\tau'$ from $\tau$ via the formula
\begin{equation} \label{eq:tau'}
\tau'(a) = \begin{cases} \tau(a), & a \in C, \\ \infty, & a \in \cA^+ \setminus C.\end{cases}
\end{equation}
If $\tau'$ is lower semi-continuous, then so is $\tau$.
\end{proposition}

\begin{proof} Observe that the set $C$ is a symmetric subcone of $\cA^+$ satisfying conditions (i) and (ii) of Lemma~\ref{lm:cone} (use Lemma~\ref{lm:traceinequality} to see that Lemma~\ref{lm:cone}~(i) holds). It therefore follows from Lemma~\ref{lm:cone} that $\cI$ is a symmetric hereditary ideal in $\cA$ and that $\cI^+ = \cI \cap \cA^+ = C$. By additivity and homogeneity of $\tau'$, its restriction to $C$ extends (uniquely) to a linear map $\tau \colon \cI \to \C$. If $a \in \cI$ is positive, then $a \in C$, so $\tau(a) = \tau'(a) \ge 0$, which shows that $\tau$ is positive. Let $x \in \cA$ be such that $x^*x \in \cI$. Then  $x^*x$ and $xx^*$ are positive elements in $\cI$, so both belong to $C$, whence $\tau(x^*x) = \tau'(x^*x) = \tau'(xx^*) = \tau(xx^*)$, so $\tau$ is a trace on $\cI$. 

It is clear that \eqref{eq:tau'} holds. If $\tau'$ is lower semi-continuous, then so is its restriction  to $C$, which shows that $\tau$ is lower semi-continuous.
\end{proof}

\noindent \emph{Whenever we talk about a trace on a \Cs{} $\cA$, we shall mean a trace defined on the cone of positive elements of that \Cs{} taking values in $[0,\infty]$, i.e., a trace in $T^+(\cA)$, and, at the same time, a linear trace on the domain defined in the proposition above, or some other domain to be specified in the context.}

As a converse to the proposition above, consider a linear trace $\tau$ defined on a hereditary symmetric ideal $\cI$ in $\cA$. Then $\tau'$ given by \eqref{eq:tau'} above, with $C= \cI \cap \cA^+$, belongs to $T^+(\cA)$, and it agrees with $\tau$ on $C$. If we apply Proposition~\ref{prop:tau->tau'} to $\tau'$, then we get back a new linear trace $\tau_0$ defined on some symmetric hereditary ideal $\cJ_0$ of $\cA$, which contains the sub-ideal $\cI_0$ of $\cI$ defined in Lemma~\ref{lm:cone} (but perhaps not $\cI$ itself); and $\tau$ and $\tau_0$ agree on $\cI_0$. 

However, this extension of $\tau$ to a trace $\tau'$ defined on the positive cone of $\cA$ is not unique, and $\tau'$ need not be lower semi-continuous, even when $\tau$ is lower semi-continuous.
If $\tau$ is lower semi-continuous, then the map $\tau' \colon \cA^+ \to [0,\infty]$ defined by
\begin{equation} \label{eq:tau}
\tau'(a) = \sup\{\tau(a_0) : 0 \le a_0 \le a, a_0 \in \cI\}, \qquad a \in \cA^+,
\end{equation}
 is a lower semi-continuous trace in $T^+(\cA)$, and it is the unique such that extends $\tau$. In the sequel, when considering a lower semi-continuous trace, we may at wish view it either as a linear trace defined on its domain, or as a trace defined on the positive cone, via \eqref{eq:tau}.

There is a canonical way of extending a lower semi-continuous trace $\tau$ defined on some hereditary symmetric ideal $\cI$ of $\cA$ to its maximal domain: first extend $\tau$ to a lower semi-continuous trace $\tau' \colon \cA^+ \to [0,\infty]$ as in \eqref{eq:tau} above; and then take the linearization $\overline{\tau}$ of $\tau'$ defined in Proposition~\ref{prop:tau->tau'}.

We quote the following well-known result for later reference, see, eg., \cite[Lemma 5.3]{RorSie:action} for a proof.

\begin{lemma} \label{prop:extending}
Let $\cA$ be a \Cs{} equipped with an action of a group $G$, and let $\tau$ be a $G$-invariant lower semi-continuous  trace on $\cA$. It follows that $\tau \circ E$ is a lower semi-continuous  trace on the (reduced) crossed product $\cA \rtimes G$ that extends $\tau$, where $E \colon \cA \rtimes G \to \cA$ is the standard conditional expectation. If $\tau$ is densely defined, then so is $\tau \circ E$. 
\end{lemma}

\noindent One should here view $\tau$ and $\tau \circ E$ as traces defined on the positive cone of $\cA$, respectively, $\cA \rtimes G$. For the claim that $\tau \circ E$ is densely defined when $\tau$ is, use that $\tau \circ E$ is finite on the positive cone of $\cP(\cA)$, and the  hereditary ideal in $\cA \rtimes G$ generated by $\cP(\cA)$ is dense in $\cA \rtimes G$. It is a curious fact that an invariant densely defined trace $\tau$ on $\cA$ need not in general extend to a trace on the crossed product $\cA \rtimes G$; in particular, $\tau \circ E$ need not be a trace if $\tau$ is not lower semi-continuous. See Example~\ref{ex:c_0}.

We end this section by considering when a \Cs{} admits a non-zero densely defined trace. 
Blackadar and Cuntz proved in \cite{BlaCuntz:infproj} 
that a stable \emph{simple} \Cs{} either contains a properly infinite projection or admits a non-zero dimension function (defined on its Pedersen ideal). In the latter case, assuming moreover that the \Cs{} is exact, it admits a non-zero densely defined trace. (This step follows from the work of Blackadar-Handelman \cite{BlaHan:quasitrace}, Haagerup, \cite{Haa:quasi}, and Kirchberg, \cite{Kir:quasitraces}, as explained in the last part of the proof of the theorem below.) Also, it is well-known that a unital exact \Cs{} admits a tracial state if and only if no matrix algebra over it is properly infinite. A common feature of simple and of unital \Cs s is that their primitive ideal spaces are compact. Recall that the primitive ideal space, $\mathrm{Prim}(\cA)$, of a \Cs{} $\cA$ is compact if and only if for all upward directed families $\{\cI_\alpha\}$ of closed two-sided ideals in $\cA$ whose union is dense in $\cA$ there is $\alpha$ such that $\cA = \cI_\alpha$. 

\begin{theorem} \label{thm:existtrace}
Let $\cA$ be an exact \Cs{} whose primitive ideal space is compact. Then $\cA$ admits a non-zero densely defined lower semi-continuous trace, i.e., $T_{\mathrm{lsc}}(\cA) \ne \{0\}$, if and only if the stabilization of $\cA$ does not contain a full properly infinite projection.
\end{theorem}

\begin{proof} The proof is most naturally phrased via dimension functions (as defined by Cuntz in \cite{Cuntz:dimension}) and the Cuntz semigroup, see, eg., \cite{CowEllIvan:Cu}. 

Observe first that the cone of densely defined lower semi-continuous traces and the primitive ideal space are not changed by stabilizing the \Cs, so may assume that $\cA$ is stable. 

The class of (closed two-sided) ideals of the form $\overline{\cI}_\cA((e-\ep)_+)$, where $e \in \cA^+$ and $\ep>0$, is upwards directed and its union is dense in $\cA$. Hence $\cA = \overline{\cI}_\cA((e-\ep_0)_+)$, for some $e \in \cA^+$ and some $\ep_0>0$ by compactness of the primitive ideal space of $\cA$. Since $(e-\ep)_+$ belongs to the Pedersen ideal and since $\cI_\cA((e-\ep)_+)$ is a dense ideal in $\cA$, for all  $0 < \ep \le \ep_0$, it follows that $\cI_\cA((e-\ep)_+) = \cP(\cA)$, for this $e \in \cA^+$ and for all $0 < \ep \le \ep_0$. Let $u_\ep = \langle (e-\ep)_+ \rangle$ be the corresponding element the Cuntz semigroup $\mathrm{Cu}(\cA)$ of $\cA$. 

Fix $0 < \ep \le \ep_0$. It follows from Corollary~\ref{cor:cone}, and the fact that $\big\langle \sum_{j=1}^n x_j^*e_j x_j \big\rangle \le \sum_{j=1}^n \langle e_j \rangle$  in $\mathrm{Cu}(\cA)$, for all positive $e_j$ and all $x_j$ in $\widetilde{\cA}$, that for each positive $a$ in $\cP(\cA)$ there exists $k \ge 1$ such that $\langle a \rangle \le k u_\ep$. In other words, $u_\ep$ is an order unit for the sub-semigroup $\mathrm{Cu}_0(\cA)$, consisting of all classes $\langle a \rangle$, where $a$ is a positive element in $\cP(\cA)$. In particular, $u_{\ep_0} \le u_\ep \le k u_{\ep_0}$, for some integer $k \ge 1$ (that depends on $\ep$). 

Consider first the case that $nu_{\ep_0}$ is properly infinite, for some integer $n \ge 1$. Upon replacing  $e$ by an $n$-fold direct sum of $e$ with itself (which is possible since $\cA$ is assumed to be stable), we may assume that $u_{\ep_0}$ itself is properly infinite, i.e., that $ku_{\ep_0} \le u_{\ep_0}$, for all integers $k \ge 1$. By the discussion in the previous paragraph, we can then conclude that $u_{\ep}$ is properly infinite and that $x \le u_{\ep}$, for all $x \in \mathrm{Cu}_0(\cA)$ and for all $0 < \ep \le \ep_0$. 

We can now follow the argument of \cite[Proposition 2.7]{PasRor:RR0}, which uses the notion of \emph{scaling elements} introduced by Blackadar and Cuntz, \cite{BlaCuntz:infproj},  to construct a full properly infinite projection $p \in \cA$: Fix $0 \le \ep < \ep_0$. Then $(e-\ep)_+$ is properly infinite, so by \cite[Proposition 3.3]{KirRor:pi} there exist positive elements $b_1,b_2$ in the hereditary sub-\Cs{} of $\cA$ generated by $(e-\ep_0)+$ such that $b_1 \perp b_1$ and $(e-\ep_0)_+ \precsim b_j$, for $j=1,2$. In particular, $b_1,b_2 \in \cP(\cA)$. As explained in \cite[Remark 2.5]{PasRor:RR0} there exists $x \in \cA$ such that $x^*x(e-\ep_0)_+ = (e-\ep_0)_+$ and $xx^*$ belongs to the hereditary sub-\Cs{} of $\cA$ generated by $b_1$. This shows that $x$ is a scaling element (cf.\ \cite[Remark 2.4]{PasRor:RR0}) satisfying $x^*xb_2 = b_2$ and $xx^*b_2 = 0$. By \cite{BlaCuntz:infproj}, see also \cite[Remark 2.4]{PasRor:RR0}, we get a projection $p \in \cA$ satisfying $b_2p=b_2$. As $u_{\ep_0} \le \langle b_2 \rangle \le \langle p \rangle \le u_{\ep_0}$, we conclude that $\langle p \rangle$ is a properly infinite order unit of $\mathrm{Cu}_0(\cA)$, whence $p$ is a full properly infinite projection in $\cA$. 

Suppose now that there is no integer $n \ge 1$ such that $nu_{\ep_0}$ is  properly infinite. We proceed to show that $T_{\mathrm{lsc}}(\cA) \ne \{0\}$ in this case. As shown above, $nu_\ep$ is not properly infinite, for any $n \ge 1$ and for any $0 < \ep \le \ep_0$. Fix $0 < \ep_1 < \ep_0$, and observe that $n u_{\ep_1} \le m u_{\ep_1}$ implies $n \le m$, for all integers $n,m \ge 0$ (since no multiple of $u_{\ep_1}$ is properly infinite). The map $f_0 \colon \N_0 u_{\ep_1} \to \R^+$, given by $f_0(nu_{\ep_1}) = n$, for $n \ge 0$, is therefore a positive additive map on the sub-semigroup $\N_0 u_{\ep_1}$ of $\mathrm{Cu}_0(\cA)$ (where $\N_0 u_{\ep_1}$ is equipped with the relative order arising from $\mathrm{Cu}_0(\cA)$). By \cite[Corollary 2.7]{BlaRor:extending} we can extend $f_0$ to a positive additive map (state) $f \colon\mathrm{Cu}_0(\cA) \to \R^+$ (since $u_{\ep_1}$ is an order unit for $\mathrm{Cu}_0(\cA)$). Let $d \colon \cP(\cA)^+ \to \R^+$ be the associated \emph{dimension function} given by $d(a) = f(\langle a \rangle)$, and let $\bar{d} \colon\mathrm{Cu}_0(\cA) \to \R^+$ be the corresponding \emph{lower semi-continuous} dimension function given by $\bar{d}(a) = \lim_{\ep > 0} d((a-\ep)_+)$, for $a \in \cP(\cA)^+$, cf. \cite[Proposition 4.1]{Ror:UHFII}. Then 
$$d((e-\ep_0)_+) \le \bar{d}((e-\ep_1)_+) \le d((e-\ep_1)_+),$$ and $d(e-\ep_0)_+)>0$ because $d$ is non-zero and $u_{\ep_0} = \langle (e-\ep_0)_+ \rangle$ is an order unit for $\mathrm{Cu}_0(\cA)$. This shows that $\bar{d}$ is non-zero.

It follows from Blackadar--Handelman, \cite[Theorem II,2,2]{BlaHan:quasitrace}, that the lower semi-con\-tin\-uous dimension function $\bar{d}$ (called rank function in \cite{BlaHan:quasitrace})  lifts to a lower semi-continuous $2$-quasitrace $\tau$ defined on the ``{pre-\Cs{}}" $\cP(\cA)$, i.e., $\bar{d} = d_\tau$, where $d_\tau(a) = \lim_{n\to\infty} \tau(a^{1/n})$, for all positive elements $a \in \cP(\cA)$. Finally, by Kirchberg's extension, \cite{Kir:quasitraces}, to the non-unital case of Haagerup's theorem, \cite{Haa:quasi}, that any $2$-quasitrace on an exact \Cs{} is a trace, $\tau$ is a non-zero lower semi-continuous densely defined trace on $\cA$.
\end{proof}

\noindent  It remains unresolved when a \Cs{} with non-compact primitive ideal space admits a non-zero densely defined lower semi-continuous trace. Clearly, $T_{\mathrm{lsc}}(\cA)$ is non-zero for all commutative \Cs s $\cA$, while the primitive ideal space of a commutative \Cs{} is compact only when it is unital. On the other hand, absence of full properly infinite projections is not sufficient to guarantee existence of non-zero lower semi-continuous traces. Take, for example, the suspension (or the cone over) any purely infinite \Cs, cf.\ \cite[Proposition 5.1]{KirRor:pi}. In \cite[Section 4]{MatuiRor:universal} it was shown that any infinite group $G$ admits a (free) action on the locally compact non-compact Cantor set ${\mathbf{K}}^*$ with no non-zero invariant Radon measures. Accordingly, $C_0({\mathbf{K}}^*) \rtimes G$ has no non-zero densely defined lower semi-continuous trace, although $C_0({\mathbf{K}}^*) \rtimes G$ admits an approximate unit consisting of projections, and, if $G$ is supramenable, eg., if $G = \Z$, then no projection in the (stabilization of) $C_0({\mathbf{K}}^*) \rtimes G$ is properly infinite. 

The latter example is covered by the proposition below. When $p$ and $q$ are projections in a \Cs{} $\cA$ and $n \ge 1$ is an integer, then denote by $p \otimes 1_n$ the $n$-fold direct sum of $p$ with itself, and write $p \prec \hspace{-.17cm} \prec q$ if $p \otimes 1_n \precsim q$, for all $n \ge 1$.

\begin{proposition} \label{prop:notrace} 
Let $\cA$ be a \Cs{} admitting an approximate unit consisting of projections. Suppose that for each projection $p \in \cA$ there exists a projection $q$ in $\cA$ with $p \prec \hspace{-.17cm} \prec q$. Then $T_{\mathrm{lsc}}(\cA) = \{0\}$. 
\end{proposition}

\begin{proof} Suppose that $\tau \in T_{\mathrm{lsc}}(\cA)$, let $p$ be a projection in $\cA$ and let $q \in \cA$ be another projection such that $p \prec \hspace{-.17cm} \prec q$. Since $p$ and $q$ belong to the Pedersen ideal of $\cA$, and hence to the domain of $\tau$, we find that $\tau(q) < \infty$, which entails that $\tau(p) = 0$. As $\cA$ has an approximate unit consisting of projection, this implies that $\tau = 0$.
\end{proof}

\noindent Here is an elementary example of an exact stably finite\footnote{A \Cs{} $\cA$ is said to be \emph{stably finite} if its stabilization $\cA \otimes \cK$ contains no infinite projections. This definition is meaningful when the \Cs{} has an approximate unit consisting of projections.}  \Cs{} satisfying the conditions of Proposition~\ref{prop:notrace}, and which accordingly admits no non-zero lower semi-continuous densely defined trace: 
Let $\cA$ be the inductive limit of the sequence $\cA_1 \to \cA_2 \to \cA_3 \to \cdots$, where $\cA_1 = \cK$, the \Cs{} of compact operators on a separable Hilbert space, where $\cA_{n+1} = \widetilde{\cA}_n \otimes \cK$, for $n \ge 1$, and where the inclusion $\cA_n \to \cA_{n+1}$ is given by $a \mapsto a \otimes e \in \cA_n \otimes \cK \subset \cA_{n+1}$, for some fixed $1$-dimensional projection $e \in \cK$.

\section{Invariant unbounded traces on \Cs s}
\label{sec:inv-traces}

\noindent We shall here use Monod's characterization of groups with the fixed-point property for cones to say something about when a (typically non-unital) \Cs{} $\cA$ with an action of a group $G$ admits an invariant trace. We are mostly interested in the existence of a (non-zero) invariant  densely defined lower semi-continuous trace, i.e., an invariant non-zero trace in the cone $T_{\mathrm{lsc}}(\cA)$ defined in Section~\ref{sec:traces-ideals}. But we shall also address the existence of more general traces (including singular traces and not densely defined traces). 

Recall from Definition~\ref{def:trace} that $T(\cI,\cA)$ is the cone of positive traces on $\cA$ with domain $\cI$, whenever $\cI$ is a hereditary symmetric ideal in $\cA$. The cone $T(\cI,\cA)$ is embedded in the complex vector space $\mathcal{L}(\cI)$ of all linear functionals on $\cI$ equipped with the locally convex weak topology induced by $\cI$. The dual space $\mathcal{L}(\cI)^*$ of $\mathcal{L}(\cI)$ is naturally isomorphic to $\cI$, cf.\ \cite[3.14]{Rudin:FunkAn} (and as remarked in \cite{Monod:cones}), i.e., $\mathcal{L}(\cI)^* = \{\varphi_a : a \in \cI\}$, where $\varphi_a$ denotes the functional $\varphi_a(\rho) = \rho(a)$, for $\rho \in \mathcal{L}(\cI)$ and $a \in \cI$. 
The dual space $\mathcal{L}(\cI)^*$ is equipped with the preordering given by $T^+(\cI,\cA)$, whereby an element $\varphi \in \mathcal{L}(\cI)^*$ is positive if $\varphi(\tau) \ge 0$, for all $\tau \in T(\cI,\cA)$. Observe that $\varphi \ge 0$ and $-\varphi \ge 0$ if and only if $\varphi(\tau)=0$, for all $\tau \in T(\cI,\cA)$. The map $a \mapsto \varphi_a$ is  a positive isomorphism, but not necessarily an order embedding, since $\varphi_a \ge 0$ does not necessarily imply that $a \ge 0$.

Monod considers real vector spaces in his paper \cite{Monod:cones}, while our vector spaces are complex by the nature of \Cs s. To translate some properties from Monod's paper to our language we shall occasionally consider the real vector space of all self-adjoint functionals $\varphi$ in  $\mathcal{L}(\cI)^*$, and we note that $\varphi_a$ is self-adjoint if and only if $a \in \cI$ is self-adjoint. 

The cone $T(\cI,\cA)$ is said to be \emph{proper} if $T(\cI,\cA) \cap -T(\cI,\cA) = \{0\}$, or, equivalently, if $0$ is the only trace in $T(\cI,\cA)$ that vanishes on $\cI \cap \cA^+$. This will hold if $\cI$ is the span of its positive elements. Most ideals considered in this paper have this property, including the Pedersen ideal $\cP(\cA)$, or any of the ideals $\cI_\cA(M)$, $\cI_\cA^G(M)$, $\cJ_\cA(M)$ or $\cJ_\cA^G(M)$, when $M$ is any non-empty subset of $\cA^+$, cf.\ Example~\ref{ex:Pedersen}, Corollary~\ref{cor:cone}, Corollary~\ref{lm:ideal} and Lemma~\ref{lm:J-positive}.

Recall also that $T_{\mathrm{lsc}}(\cA) = T(\cP(\cA),\cA)$. We allow for the possibility that  the cones $T(\cI,\cA)$ and $T_{\mathrm{lsc}}(\cA)$ are trivial, that is, equal to $\{0\}$, unless otherwise stated.

\begin{proposition} \label{prop:2}
For each \Cs{} $\cA$ and for each hereditary symmetric ideal $\cI$ in $\cA$, the cone $T(\cI,\cA)$ is weakly complete.  In particular, $T_{\mathrm{lsc}}(\cA)$ is weakly complete. 
\end{proposition}

\begin{proof} We must show that each weak Cauchy net in $T(\cI,\cA)$ is weakly convergent, i.e., if $(\tau_i)_i$ is a net in $T(\cI,\cA)$ such that $(\varphi(\tau_i))_i$ is Cauchy in $\C$, for all $\varphi \in \mathcal{L}(\cI)^*$, then the net converges weakly in $T(\cI,\cA)$. Since $\varphi_a(\tau_i) = \tau_i(a)$, for all $a \in \cI$, being weakly Cauchy implies that $(\tau_i(a))_i$ is Cauchy and hence convergent in $\C$, for all $a \in \cI$. Set $\tau(a) = \lim_i \tau_i(a)$, for all $a \in \cI$. It is easy to check that $\tau \colon \cI \to \C$ is in fact a trace, so it belongs to $T(\cI,\cA)$, and since $\varphi_a(\tau_i) \to \varphi_a(\tau)$, for all $a \in \cI$,  $\tau$ is the weak limit of the net $(\tau_i)_i$, as desired.
\end{proof}

\noindent Consider an action $\alpha$ of a (discrete) group $G$ on $\cA$. If $\cI$ is a $G$-invariant hereditary symmetric ideal in $\cA$, then $G$ induces an action of the cone $T(\cI,\cA)$ by $t.\tau = \tau \circ \alpha_t^{-1}$, for $t \in G$ and $\tau \in T(\cI,\cA)$. It is clear that this action of $G$ on $T(\cI,\cA)$ is continuous. Each automorphism of $\cA$ leaves the Pedersen ideal  invariant, so each group action on $\cA$ induces an action on the cone $T_{\mathrm{lsc}}(\cA)$.

The action of $G$ on $T(\cI,\cA)$ is in \cite{Monod:cones} said to be of \emph{cobounded type} if there exists a positive functional $\varphi$ in $\mathcal{L}(\cI)^*$ which $G$-dominates\footnote{If $\varphi$ and $\psi$ are self-adjoint functionals in $\mathcal{L}(\cI)^*$, then $\psi$ is $G$-dominated by $\varphi$ if $\psi \le \sum_{j=1}^n t_j.\varphi$, for some $n \ge 1$ and some $t_1,t_2, \dots, t_n \in G$.} any other self-adjoint functional in $\mathcal{L}(\cI)^*$.  
This condition is automatically satisfied when $\cI = \cJ_\cA^G(e)$, for some positive element $e \in \cA^+$, cf.\ Corollary~\ref{cor:cobdd} below, but not always when $\cI$ is the Pedersen ideal of $\cA$. However, in the latter case we can reformulate the coboundedness condition into more familiar statements for \Cs s.

A  positive element $e \in \cI$ is said to \emph{$G$-dominate} a self-adjoint element $a \in \cI$ if there are group elements $t_1, \dots, t_n$ such that $a \le \sum_{j=1}^n \alpha_{t_j}(e)$; and $e$ is said to \emph{tracially $G$-dominate} $a$ if there are group elements $t_1, \dots, t_n$ such that $\tau(a) \le \sum_{j=1}^n \tau(\alpha_{t_j}(e))$, for all $\tau \in T(\cI,\cA)$. The latter holds if and only if $\varphi_a \le \sum_{j=1}^n t_j^{-1}. \varphi_e$; in other words, if $\varphi_e$ $G$-dominates $\varphi_a$. We can summarize these remarks  as follows:

\begin{lemma} \label{lm:cobounded0}
Let $\cA$ be a \Cs{} equipped with  an action of a group $G$, and let $\cI$ be an invariant hereditary symmetric ideal in $\cA$. The induced action of $G$ on the cone $T(\cI,\cA)$ is  of cobounded type if and only if there is a positive element $e \in \cI$, which tracially $G$-dominates each self-adjoint element  $a \in \cI$.
\end{lemma}

\begin{lemma} 
\label{lm:cobounded1a}
Let $\cI$ be a $G$-invariant hereditary symmetric ideal in a \Cs{} $\cA$, and let $e$ be a positive element in $\cI$. Then the functional $\varphi_a$ is $G$-dominated by $\varphi_e$, for each self-adjoint element $a \in \cJ_\cA^G(e)$.
\end{lemma}

\begin{proof} Let $C$ be the set of positive elements $a \in \cI$ such that $\varphi_a$ is $G$-dominated by $e$. We claim that $C$ is a $G$-invariant symmetric cone in $\cA^+$, which satisfies (i) and (ii) of Lemma~\ref{lm:cone}. Since $e$ clearly belongs to $C$, it will then follow from Lemma~\ref{lm:cone} that $\cJ_\cA^G(e) \cap \cA^+ \subseteq C$, and this will prove the lemma.

The set of positive $\varphi \in \mathcal{L}(\cI)^*$ that are $G$-dominated by $\varphi_e$ is a $G$-invariant hereditary cone in the positive cone of $\mathcal{L}(\cI)^*$. As the map $a \mapsto \varphi_a$ is linear, order preserving and satisfies $\varphi_{\alpha_t(a)} = t.\varphi_a$, for $a \in \cI^+$ and $t \in G$, we conclude that $C$ is a hereditary $G$-invariant cone in $\cA^+$. For each $x \in \cA$, for which $x^*x$ (and hence $xx^*$) belong to $\cI$, we have $\varphi_{x^*x} \le \varphi_{xx^*} \le \varphi_{x^*x}$, which implies that $C$ is symmetric.  It remains to show that $x^*ax$ belongs to $C$ when $a$ belongs to $C$ and $x$ belongs to $\widetilde{\cA}$.  To see this, recall from Lemma~\ref{lm:traceinequality} that $\tau(x^*ax) \le \|x\|^2 \tau(a)$, so $\varphi_{x^*ax} \le \|x\|^2 \varphi_a$, and the latter is $G$-dominated by $\varphi_e$ since $a \in C$ (and since $C$ is a cone).
\end{proof}

\noindent The corollary below follows immediately from Lemma~\ref{lm:cobounded1a}.

\begin{corollary} \label{cor:cobdd} The action of a group $G$ on the cone $T(\cJ_\cA^G(e), \cA)$ is of cobounded type whenever $\cA$ is a \Cs{} with an action of $G$ and  $e$ is a positive element in $\cA$.
\end{corollary}

\noindent The action of a group on the Pedersen ideal of a \Cs{} is not always of cobounded type, as  illustrated in the proposition below, that covers the case of commutative \Cs s, and which paraphrases and expands a remark on page 71 in \cite{Monod:cones}. We remind the reader that the action of a group $G$ on a locally compact Hausdorff space is \emph{co-compact} if $X = G.K$, for some compact subset $K$ of $X$.

\begin{proposition} \label{lm:cocompact} Let $X$ be a locally compact Hausdorff space equipped with a continuous action of a group $G$. Then the following conditions are equivalent:
\begin{enumerate}
\item The action of $G$ on $X$ is co-compact.
\item $X$ is compact in the (non-Hausdorff) topology on $X$ consisting of the $G$-invariant open subsets of $X$.
\item The action of $G$ on the cone of Radon measures on $X$ equipped with the vague topology is of cobounded type.
\item The action of $G$ on the cone $T_{\mathrm{lsc}}(C_0(X))$ is of cobounded type.
\end{enumerate}
\end{proposition}

\begin{proof} (i) $\Rightarrow$ (ii). Let $K$ be a compact subset of $X$ witnessing co-compactness of the action.  Let $\{U_i\}_{i \in I}$ be a collection of invariant open sets that covers $X$. Select a finite subset $F\subseteq I$ such that $\{U_i\}_{i \in F}$ covers $K$. Then $\bigcup_{i \in F} U_i = X$, being a $G$-invariant set that contains $K$.

(ii) $\Rightarrow$ (i). Let $\{U_i\}_{i \in I}$ be the collection of all open pre-compact subsets of $X$. Then $X = \bigcup_{i \in I} U_i$, because $X$ is locally compact. For each $i \in I$, set $V_i = \bigcup_{t \in G} t.U_i$. The families $\{U_i\}_{i \in I}$ and $\{V_i\}_{i \in I}$ are both upwards directed (both are closed under forming finite unions). It follows by compactness of $X$ in the topology of invariant open sets that $X = V_i$, for some $i \in I$. Hence $X = \bigcup_{t \in G} t.K$, when $K$ is the (compact) closure of $U_i$. 

(i) $\Rightarrow$ (iv). The cone $T_{\mathrm{lsc}}(C_0(X))$ is embedded into the vector space $\mathcal{L}(C_c(X))$ equipped with the weak topology from $C_c(X)$; and the dual space, $\mathcal{L}(C_c(X))^*$, is equal to $C_c(X)$. By co-compactness of the action we can find sets $U \subseteq K \subseteq X$,  such that $K$ is compact, $U$ is open, and $G.U = X$. Let $f \in C_c(X)$ be such that $1_K \le f \le 1$. Then any real valued function $g \in C_c(X)$ is $G$-dominated by $f$. Indeed, if $F$ is a finite subset of $G$ such that the support of $g$ is contained in $\bigcup_{t \in F} t.U$, then  $g \le \|g\|_\infty \sum_{t \in F} t.f$.

(iv) $\Rightarrow$ (i). Following the set-up of the proof above we can find a positive function $f \in C_c(X)$ which $G$-dominates any other real valued function in $C_c(X)$. Let $K$ be the support of $f$. Let $x \in X$ and choose a positive function $g \in C_c(X)$ such that $g(x)=1$. Then $g \le \sum_{t \in F} t.f$ for some finite subset $F$ of $G$. Thus $t.x \in K$, for some $t \in K$. This shows that $X = G.K$. 

(iii) $\Leftrightarrow$ (iv). By Riesz' theorem there is a one-to-one correspondance between Radon measures on $X$ and positive linear functionals (hence traces) on $C_c(X) = \cP(C_0(X))$. 
\end{proof}

\noindent We have previously, in Theorem~\ref{thm:existtrace}, considered \Cs s $\cA$ whose primitive ideal space, $\mathrm{Prim}(\cA)$, is compact, which happens if  whenever $\{\cI_\alpha\}$ is an upwards directed net of closed two-sided ideals in $\cA$ such that $\bigcup_\alpha \cI_\alpha$ is dense in $\cA$, then $\cA=\cI_\alpha$, for some $\alpha$. Let us say that a \Cs{} $\cA$ equipped with an action of a group $G$ is \emph{$G$-compact} if whenever $\{\cI_\alpha\}$ is an upwards directed net of $G$-invariant closed two-sided ideals in $\cA$ such that $\bigcup_\alpha \cI_\alpha$ is dense in $\cA$, then $\cA=\cI_\alpha$, for some $\alpha$. In the commutative case this is equivalent to condition (ii) of Proposition~\ref{lm:cocompact}. In the non-commutative case we have the following:

\begin{proposition} \label{prop:cobdd}
Let $\cA$ be a \Cs{} with an action of a group $G$, and suppose that $\cA$ is $G$-compact. 
Then the induced action of $G$ on the cone $T_{\mathrm{lsc}}(\cA)$ is of cobounded type. 
\end{proposition}

\begin{proof}  The family $\{\overline{\cJ_\cA^G}(e)\}$ of $G$-invariant ideals in $\cA$, where $e \in \cP(\cA)^+$, is upward directed and its union is dense in $\cA$. Hence by the $G$-compactness assumption there is a positive element  $e$ in $\cP(\cA)$ such that the hereditary ideal $\cJ_\cA^G(e)$ is dense in $\cA$. As $e \in \cP(\cA)$, this entails that  $\cJ_\cA^G(e)= \cP(\cA)$. We can therefore conclude from Lemma~\ref{lm:cobounded1a} that each  self-adjoint $\varphi \in \mathcal{L}(\cP(\cA))^*$ is $G$-dominated by $\varphi_e$.
\end{proof}

\noindent Examples of $G$-compact \Cs s include any simple \Cs{} and, more generally, any \Cs{} with no non-trivial $G$-invariant closed two-sided ideals. Also any \Cs{} $\cA$, which contains a $G$-full projection, i.e., a projection $p$ such that $\overline{\cI_\cA^G}(p) = \cA$, is $G$-compact.

It is not hard to show that the conclusion of Proposition~\ref{prop:cobdd} still holds under the weaker assumption that $\cA/\cI_0$ is $G$-compact, where $\cI_0$ is the closed two-sided $G$-invariant ideal
$$\cI_0 = \bigcap_{\tau \in T_{\mathrm{lsc}}(\cA)} \{x \in \cA : \tau(x^*x) = 0\}.$$

Following Monod, \cite{Monod:cones}, a subset $M$ of $T(\cI,\cA)$ is \emph{bounded} if for each open neighborhood $0 \in \cU \subseteq \mathcal{L}(\cI)$  there is $r > 0$ such that $M \subseteq r \, \cU$. 

\begin{lemma} \label{lm:3} Let $\cA$ be a \Cs{} and let $\cI$ be a hereditary symmetric ideal in $\cA$.
A subset $M$ of $T(\cI,\cA)$ is bounded if and only if $\{\tau(a) : \tau \in M\}$ is bounded for all $a \in \cI$. In particular,  $M \subseteq T_{\mathrm{lsc}}(\cA)$ is bounded if and only if $\{\tau(a) : \tau \in M\}$ is bounded for all $a \in \cP(\cA)$.
\end{lemma}

\begin{proof} For each self-adjoint element $a \in \cI$, the set $\cU_a := 
\{\varphi \in \mathcal{L}(\cI) : |\varphi(a)| < 1\}$ is an open neighborhood of $0$. Hence, if $M$ is bounded, then $M \subseteq r \, \cU_a$ for some $r >0$, which  entails that $|\tau(a)| < r$, for all $\tau \in M$. Conversely, suppose that 
$\{\tau(a) : \tau \in M\}$ is bounded, for all $a \in \cI$, and let $\cU$ be an open neighborhood of $0$. Then there are $a_1, \dots, a_n$ in $\cI$ such that $\cU_{a_1, \dots, a_n} \subseteq \cU$, where 
$$\cU_{a_1, \dots, a_n}  = \{\varphi \in \mathcal{L}(\cI) : |\varphi(a_i)| < 1, \; \text{for} \: i=1,2, \dots, n\}.$$
Set $r_i = \sup \{|\tau(a_i)| : \tau \in M\}$ and set $r = 1+ \max_i r_i$. Then $|\tau(a_i)| < r$, for all $i$ and for all $\tau \in M$, whence $M \subseteq r \, \cU_{a_1,a_2, \dots, a_n} \subseteq r\, \cU$, as desired. 
\end{proof}

\noindent A trace $\tau \in T(\cI,\cA)$ will be said to be \emph{locally bounded} if it is bounded on the $G$-orbit of each $a \in \cI$. Clearly, any bounded trace is $G$-bounded (regardless of the properties of the action). By Monod, \cite{Monod:cones}, the action of $G$ on $T(\cI,\cA)$ is \emph{locally bounded} if $T(\cI,\cA)$ contains a non-zero bounded orbit.  

\begin{lemma} \label{lm:locallybounded} Let $G$ be a group acting on a \Cs{} $\cA$,  let $\cI$ be a hereditary symmetric invariant ideal in $\cA$, and let $e$ be a positive element in $\cA$.
\begin{enumerate}
\item The induced action of $G$ on $T(\cI,\cA)$ is locally bounded if and only if there is a non-zero locally bounded trace in $T(\cI,\cA)$.
\item A trace in $T(\cJ_\cA^G(e),\cA)$ is locally bounded if it is bounded on the orbit of $e$.
\end{enumerate}
\end{lemma}

\begin{proof} (i) is an immediate reformulation of Lemma~\ref{lm:3}. To prove (ii), let $\tau \in T(\cJ_\cA^G(e),\cA)$ be bounded on the orbit of $e$, and let
$C$ be the set of positive elements $a  \in \cJ_\cA^G(e)$ such that $\tau$ is bounded on the orbit of $a$. Then $C$ satisfies conditions (i) and (ii) of Lemma~\ref{lm:cone}, cf.\ Lemma~\ref{lm:traceinequality}, $e \in C$, and $C$ is $G$-invariant, so $C = \cJ_\cA^G(e) \cap \cA^+$ by Lemma~\ref{lm:cone}.
\end{proof}

\noindent Groups with the fixed-point property for cones allow invariant lower semi-continuous densely defined traces on a \Cs{} as follows:

\begin{theorem} \label{thm:a} Let $\cA$ be a \Cs{} with $T_{\mathrm{lsc}}(\cA) \ne \{0\}$, and let $G$ be a group with the fixed-point property for cones, which acts on $\cA$ making $\cA$ $G$-compact. Then there is a non-zero invariant, necessarily lower semi-continuous, trace in $T_{\mathrm{lsc}}(\cA)$  if and only if there is a non-zero  locally bounded trace in $T_{\mathrm{lsc}}(\cA)$. 

If this is the case, then $T_{\mathrm{lsc}}(\cA \rtimes G)$ is non-zero, i.e., $\cA \rtimes G$ admits a non-zero lower semi-continuous densely defined trace.
\end{theorem}

\begin{proof} The cone $T_{\mathrm{lsc}}(\cA)$ is weakly complete by Proposition~\ref{prop:2}, the action of $G$ on  $T_{\mathrm{lsc}}(\cA)$ is affine and continuous. By the assumption on $G$ there is a non-zero invariant trace in $T_{\mathrm{lsc}}(\cA)$ if the action of $G$ on $T_{\mathrm{lsc}}(\cA)$ is of cobounded type and locally bounded. The former holds by Proposition~\ref{prop:cobdd}, since $A$ is assumed to be $G$-compact, and the latter holds by Lemma~\ref{lm:locallybounded} if there is a non-zero locally bounded trace.  Conversely, any invariant trace in $T_{\mathrm{lsc}}(\cA)$ is trivially $G$-bounded. The last claim follows from Lemma~\ref{prop:extending}.
\end{proof}

\noindent
Consider the class of groups $G$ for which Theorem~\ref{thm:a} holds. This class contains the class of groups with the fixed-point propety for cones, and it is contained in the class of supramenable groups, cf.\ the proposition below, that paraphrases \cite[Theorem 1.1]{KelMonRor:supra}. We do not know if Theorem~\ref{thm:a} characterizes groups with the fixed-point property for cones, or if it characterizes the class of supramenable groups, or some intermediate class of groups. 


\begin{proposition} \label{prop:KMR}
The following conditions are equivalent for each group $G$:
\begin{enumerate}
\item $G$ is supramenable.
\item Whenever $G$ acts on a commutative \Cs{} $\cA$, such that $\cA$ is $G$-compact, then there is a non-zero invariant trace in $T_{\mathrm{lsc}}(\cA)$.
\item Whenever $G$ acts on a commutative \Cs{} $\cA$, then for each projection $p \in \cA$ there is an invariant lower semi-continuous trace $\tau \in T^+(\cA)$ with $\tau(p)=1$.
\item Whenever $G$ acts on a commutative \Cs{} $\cA$,  then for each projection $p \in \cA$ there is a
lower semi-continuous trace $\tau \in T^+(\cA \rtimes G)$ with $\tau(p)=1$.
\end{enumerate}
\end{proposition}

\begin{proof}  (i) $\Rightarrow$ (ii). Let $X$ be the spectrum of $\cA$, so that $\cA = C_0(X)$. If $\cA$ is $G$-compact, then the action of $G$ on $X$ is co-compact, cf.\ Proposition~\ref{lm:cocompact}, so by  \cite[Theorem 1.1]{KelMonRor:supra} there is a non-zero invariant Radon measure on $X$, since $G$ is supramenable. Integrating with respect to this measure gives a non-zero invariant trace in $T_{\mathrm{lsc}}(\cA)$.

(ii) $\Rightarrow$ (iii). The \Cs{} $\cB = \overline{\cI_\cA^G}(p)$ is $G$-compact being generated by a projection, so (ii) implies that there is a non-zero invariant  $\tau \in T_{\mathrm{lsc}}(\cB)$. We must show that $0 < \tau(p) < \infty$. The latter inequality holds because
$p \in \cP(\cB)$. If $\tau(p)=0$, then $\tau(x) = 0$ for all $x \in \cI_\cB^G(p)$, and
 $\cI_\cB^G(p) = \cP(\cB)$, cf.\ Lemma~\ref{lm:symmetric-proj},  contradicting that $\tau\ne 0$. Finally, we can extend $\tau$ to a  lower semi-continuous trace in $T^+(\cA)$  by \eqref{eq:tau}.

(iii) $\Rightarrow$ (iv). This follows from Lemma~\ref{prop:extending}.

(iv) $\Rightarrow$ (i). If $G$ is non-supramenable, then, by
\cite[Theorem 1.1]{KelMonRor:supra}, it admits a minimal, free, purely infinite action on the locally compact non-compact Cantor set $\mathbf{K}^*$, making the crossed product $C_0(\mathbf{K}^*) \rtimes G$ purely infinite (and simple). In particular, there is no non-zero lower semi-continuous trace on $C_0(\mathbf{K}^*) \rtimes G$. 
\end{proof}

\begin{remark} \label{rem:unital} If $\tau$ is a trace on a \Cs{} $\cB$ such that $\tau(p)=1$, for some projection $p \in \cB$, then the restriction of $\tau$ to the unital corner \Cs{} $p\cB p$ is a tracial state. Conversely, each tracial state $\tau$  on $p\cB p$ extends to a lower semi-continuous trace in $T^+(\cB)$ normalizing $p$.

To see this, consider the family $\{\mathcal{H}_\alpha\}_\alpha$ of all $\sigma$-unital hereditary sub-\Cs s of $\overline{\cI_\cB}(p)$ containing $p\cB p$, and observe that $p\cB p$ is a full hereditary sub-\Cs{} of $\mathcal{H}_\alpha$, for each $\alpha$. By Brown's theorem, the  inclusion $p\cB p \to p \cB p \otimes \cK$, where $\cK$ is the \Cs{} of compact operators on a separable Hilbert space, extends to an inclusion  $\mathcal{H}_\alpha \to p \cB p \otimes \cK$, for each $\alpha$. The tracical state $\tau$ on $p\cB p$ extends (uniquely) to a lower semi-continuous trace on the positive cone of $p \cB p \otimes \cK$, and thus restricts to a lower semi-continuous trace $\tau_\alpha$ on the positive cone of $\mathcal{H}_\alpha$. The extension of $\tau$ to $\tau_\alpha$ on $\mathcal{H}_\alpha$ is unique.

The familiy $\{\mathcal{H}_\alpha\}_\alpha$ is upwards directed with union $\bigcup_\alpha \mathcal{H}_\alpha = \overline{\cI_\cB}(p)$. To see the first claim, recall that a \Cs{} is $\sigma$-unital precisely if it contains a strictly positive element. Let $b_\alpha \in \mathcal{H}_\alpha$ be strictly positive. Then $b_\alpha + b_\beta$ is a strictly positive element in $\overline{(b_\alpha + b_\beta)\cB(b_\alpha + b_\beta)}$, and the latter is therefore a $\sigma$-unital sub-\Cs s of $\overline{\cI_\cB}(p)$ containing $\mathcal{H}_\alpha$ and $\mathcal{H}_\beta$. For the second claim, if $a \in \overline{\cI_\cB}(p)$ is positive, then $\overline{(a+p)\cB(a+p)}$ is a $\sigma$-unital hereditary sub-\Cs{} of $\overline{\cI_\cB}(p)$ containing $p\cB p$ and $a$. There exists therefore a (unique) lower semi-continuous trace $\tau$ in $T^+(\overline{\cI_\cB}(p))$ that extends each $\tau_\alpha$, and hence $\tau$.

Extend $\tau$ further to a lower semi-continuous trace in $T^+(\cB)$ using  \eqref{eq:tau}. Finally, if we wish, we can linearize $\tau$  using Proposition~\ref{prop:tau->tau'}.
 \end{remark}

\noindent
If we depart from lower semi-continuous densely defined traces and consider possibly singular traces, then we do obtain a $C^*$-algebraic characterization of groups with the fixed-point property for cones. The theorem below is a non-commutative analog of the equivalence between (1) and (4) of Theorem 7 in \cite{Monod:cones}.

\begin{theorem}  \label{thm:b}
The following conditions are equivalent for any group $G$:
\begin{enumerate}
\item $G$ has the fixed-point property for cones.
\item Whenever $G$ acts on a \Cs{} $\cA$ and whenever $e$ is a positive element in $\cA$ for which there exists a non-zero trace in $T(\cJ_\cA^G(e),\cA)$, which is bounded on the $G$-orbit of $e$, then there exists an invariant trace in $T(\cJ_\cA^G(e),\cA)$ normalizing $e$.
\item Whenever $G$ acts on a \Cs{} $\cA$  and whenever $e$ is a positive element in $\cA$ for which there exists  a trace in $T^+(\cA)$  which is non-zero and bounded on the $G$-orbit of $e$, then there is an invariant trace in $T^+(\cA)$ normalizing $e$. 
\end{enumerate}
\end{theorem}

\begin{proof}
(i) $\Rightarrow$ (ii). Set $\cI = \cJ_\cA^G(e)$. Then $T(\cI,\cA)$ is weakly complete by Proposition~\ref{prop:2}, the action of $G$ on $T(\cI,\cA)$ is locally bounded, by Lemma~\ref{lm:locallybounded}, and of cobounded type by Corollary~\ref{cor:cobdd}. The cone of traces $T(\cI,\cA)$ therefore has a non-zero fixed point $\tau_0$ since $G$ has the fixed-point property for cones. As $\tau_0$ is non-zero and invariant we can use Lemma~\ref{lm:traceinequality}, Lemma~\ref{lm:ideal} and the fact that $e \in \cI$  to conclude that $0 < \tau_0(e)  <\infty$.

(ii) $\Rightarrow$ (i). Property (ii), applied to $\cA = \ell^\infty(G)$ (see also Section~\ref{sec:inv-integrals} below), says that condition (4) in \cite[Theorem 7]{Monod:cones} holds, which again, by that theorem, is equivalent to (i).

(ii) $\Leftrightarrow$ (iii). This follows from Proposition~\ref{prop:tau->tau'} and the remarks below that proposition.
\end{proof}

\noindent  If $\cA$ is commutative, or, more generally, if $\cA$ admits a separating family of bounded traces, then the condition in Theorem~\ref{thm:b} (ii) and (iii), that there
exists a non-zero locally bounded traces, is always satisfied independent on the action of $G$ on $\cA$.

Below we specialize Theorem~\ref{thm:b} to the case where the positive element   is a projection.

\begin{corollary} \label{cor:b}
Let $G$ be a group with the fixed-point property for cones acting on a \Cs{} $\cA$, and let $p \in \cA$ be a non-zero projection. The following conditions are equivalent:
\begin{enumerate}
\item There is a non-zero lower semi-continuous  trace with domain $\cJ_\cA^G(p)$, which is  bounded on the $G$-orbit  of $p$.
\item There is an invariant lower semi-continuous trace in $T^+(\cA)$ normalized on $p$.
\item There is a lower semi-continuous trace in $T^+(\cA \rtimes G)$ normalized on $p$. 
\end{enumerate}
If, in addition, $\cA \rtimes G$ is exact, then the conditions above are equivalent to:
\begin{itemize}
\item[\rm{(iv)}] The \Cs{} $p(\cA \rtimes G)p \otimes M_n$ is not properly infinite, for all $n \ge 1$.
\end{itemize}
\end{corollary}

\begin{proof}  (i) $\Rightarrow$ (ii). It follows from Theorem~\ref{thm:b} that there is an invariant trace $\tau$ with domain $\cJ_\cA^G(p)$ normalizing $p$. As remarked below Theorem~\ref{thm:GKP}, each trace on $\cJ_\cA^G(p)$ is automatically lower semi-continuous.

(ii) $\Rightarrow$ (iii). Any invariant lower semi-continuous trace $\tau$ on $\cA$ with $\tau(p)=1$ extends to a lower semi-continuous trace on $\cA \rtimes G$, by  Lemma~\ref{prop:extending}.

(iii) $\Rightarrow$ (i). Take the restriction to $\cA$ of the trace whose existence is claimed in (iii) and linearize as in Proposition~\ref{prop:tau->tau'}.

(iii) $\Leftrightarrow$ (iv). It is well-known that this equivalence holds when $\cA \rtimes G$ is exact, see Remark~\ref{rem:unital} and the comments above Theorem~\ref{thm:existtrace}.
\end{proof}

\noindent We can rephrase the corollary above as follows. A projection $p$ in a \Cs{} $\cA$ with a $G$-action can fail to be normalized by an invariant trace on $\cA$ for two reasons. Either $\cA$ does not have a trace that is non-zero and bounded on the $G$-orbit of $p$, or the group $G$ possesses some amount of ``paradoxicality'' materialized in failing to have the fixed-point property for cones. In Lemma~\ref{lm:subexpgrowth} we give non-obvious examples of projections for which there exist a trace that is non-zero and bounded on the $G$-orbit of the projection.

While we do not know that every group without the fixed-point property for cones has the ability of acting on a \Cs{} $\cA$ in such a way that it obstructs the existence of an invariant trace normalizing a given projection $p$ in $\cA$, for which there is a trace on $\cA$ that is non-zero and bounded on the orbit of $p$, it does follow from Proposition~\ref{prop:KMR} that all non-supramenable groups have this quality. 

In conclusion, we do not know if Corollary~\ref{cor:b} characterizes the class of groups with the fixed-point property for cones, the class of supramenable groups, or some intermediate class of groups.

\section{Invariant integrals on $\ell^\infty(G)$}
\label{sec:inv-integrals}

\noindent
The bounded (complex valued) functions,  $\ell^\infty(G)$, on a (discrete) group $G$ is a unital \Cs{} equipped with an action of the group $G$ by left-translation. Following the notation of Monod, \cite{Monod:cones}, for each positive $f \in \ell^\infty(G)$, let $\ell^\infty(G,f)$ denote the set of bounded functions $g$ on $G$ whose absolute value is \emph{$G$-bounded} by $f$, i.e., for which $|g| \le \sum_{j=1}^n t_j. f$, for some $n \ge 1$ and some $t_1, \dots, t_n \in G$, where $t.f$ is the left-translate of $f \in \ell^\infty(G)$ by $t \in G$.  In the language of \Cs s, $\ell^\infty(G,f)$ is the smallest (automatically symmetric) hereditary $G$-invariant ideal in $\ell^\infty(G)$ containing $f$, denoted by $\cI_{\ell^\infty(G)}^G(f)$ in the previous sections.  The Pedersen ideal of the uniform closure of $\ell^\infty(G,f)$ is denoted by $\ell_c^\infty(G,f)$, and we have the following inclusions:
$$\ell_c^\infty(G,f) \subseteq \ell^\infty(G,f) \subseteq \overline{\ell^\infty(G,f)}.$$
An \emph{invariant integral on $G$ normalized for $f$} is a $G$-invariant positive linear functional $\mu$ on $\ell^\infty(G,f)$ satisfying $\mu(f)=1$. In the language of \Cs s, an (invariant) integral on $\ell^\infty(G,f)$ is an (invariant) trace on $\ell^\infty(G)$ with domain $\ell^\infty(G,f)$.

Monod observed in \cite[Theorem 7]{Monod:cones} that $G$ has the fixed-point property for cones if and only if  for each positive function $f$ in $\ell^\infty(G)$ there is an invariant integral on $G$ normalized for $f$. (This result is extended to general \Cs s in our Theorem~\ref{thm:b}.)

Let $\mu$ be an invariant integral on $\ell^\infty(G)$ normalized on some positive function $f \in \ell^\infty(G)$. We say that $\mu$ is \emph{lower semi-continuous} if $\mu(h) = \lim_{n \to \infty}\mu(h_n)$, whenever $\{h_n\}_{n=1}^\infty$ is an increasing seqence of positive functions in $\ell^\infty(G,f)$ converging uniformly to $h \in \ell^\infty(G,f)$, cf.\ Lemma~\ref{lm:lsc-eq}.
The restriction of $\mu$ to $\ell_c^\infty(G,f)$ is automatically lower semi-continuous by Theorem~\ref{thm:GKP} (Pedersen). If the restriction of $\mu$ to $\ell_c^\infty(G,f)$ is zero, then $\mu$ is said to be \emph{singular}. 

We shall in this section find conditions that will ensure that an (invariant) integral is lower semi-continuous and also exhibit situations where such integrals necessarily are singular. We start by rephrasing what Proposition~\ref{prop:KMR} says about the existence of lower semi-continuous integrals:

\begin{proposition} \label{lm:cocpt-f}
Let $G$ be a supramenable group, let $f$ be a positive function in $\ell^\infty(G)$, and suppose that  $\overline{\ell^\infty(G,f)}$ is $G$-compact (see above Proposition~\ref{prop:cobdd}). Then there is a non-zero invariant lower semi-continuous integral on $\ell^\infty_c(G,f)$.
\end{proposition}

\begin{example} \label{ex:compacttype} Let $G$ be a group and let $f$ be a positive function in $\ell^\infty(G)$. Then $\overline{\ell^\infty(G,f)}$ is $G$-compact if and only if there exists $\delta>0$ such that $A(f,\ep) \propto_G A(f,\delta)$, for all $0 < \ep < \delta$; where $A(f,\eta) = \{t \in G : f(t) > \eta\}$, and where $A \propto_G B$, for subsets $A,B$ of $G$, means that $A \subseteq \bigcup_{t \in F} tB$, for some finite subset $F$ of $G$. We will not go into the details of the proof of this, but just mention that to prove the ``if'' part, one observes
that $A(f,\ep) \propto_G A(f,\delta)$ implies that $(f-\ep)_+ \in \cI_{\ell^\infty(G)}^G((f-\delta')_+)$, for all $0 < \delta' < \delta$.

The condition above ensuring $G$-compactness of $\overline{\ell^\infty(G,f)}$ can further be rewritten as follows:  set $A_n = \{t \in G: \frac{1}{n} < f(t) \le \frac{1}{n-1}\}$, for all $n \ge 1$. Then $\overline{\ell^\infty(G,f)}$ is $G$-compact if and only if there exists $N_0 \ge 1$ such that $\bigcup_{n=1}^N A_n \propto_G \bigcup_{n=1}^{N_0} A_n$, for all $N \ge N_0$. This condition holds if $A_k \propto_G \bigcup_{n=1}^{N_0} A_n = A(f,1/N_0)$, for all $k > N_0$. In other words, $\overline{\ell^\infty(G,f)}$ is $G$-compact if the set of $t \in G$ where $f(t)$ is ``very small'' can be controlled by the set where $f(t)$ is ``small enough''.
\end{example}

\noindent Proposition~\ref{lm:cocpt-f}  and the example above does not give information about the existence of invariant lower semi-continuous integrals normalizing the given positive function $f \in \ell^\infty(G)$. See Example~\ref{ex:c_0}  below for more about this problem. 
Using an example from \cite{MatuiRor:universal}, we proceed to show that the conclusion of Proposition~\ref{lm:cocpt-f} fails without the assumption on $G$-compactness.

\begin{proposition} \label{prop:non-cocpt}
For each countably infinite group $G$ there is a positive function $f \in \ell^\infty(G)$ such that $\ell^\infty_c(G,f)$ admits no non-zero invariant integral. In particular, if an invariant integral on $G$ normalized on $f$ exists, as is the case whenever $G$ has the fixed-point property for cones, then it is singular.
\end{proposition}

\begin{proof} Let $G$ be a countably infinite group. By \cite[Proposition 4.3]{MatuiRor:universal} and its proof there is an increasing sequence $\{A_n\}_{n \ge 1}$ of (infinite) subsets of $G$ such that if $K_n$ is the (compact-open) closure of (the open set) $A_n$ in $\beta G$, and if $X_n = \bigcup_{t \in G} t.K_n$, then $X := \bigcup_{n \ge 1} X_n$  is an open $G$-invariant subset of $\beta G$ that admits no non-zero invariant Radon measure. 

Let $\varphi \colon \ell^\infty(G) \to C(\beta G)$ be the canonical $^*$-isomorphism of \Cs s, and let $\cI$ be the closed $G$-invariant ideal of $\ell^\infty(G)$ such that $\varphi(\cI) = C_0(X)$. Observe that $C_0(X)$ is the closed $G$-invariant ideal in $C(\beta G)$ generated by the indicator functions $1_{K_n}$, $n \ge 1$. Since $\varphi(1_{A_n}) = 1_{K_n}$, we conclude that $\cI$ is the closed $G$-invariant ideal in $\ell^\infty(G)$ generated by the projections $1_{A_n}$, $n \ge 1$, or by the positive function $f = \sum_{n \ge 1} n^{-2} 1_{A_n} \in \ell^\infty(G)$. Hence $\cI = \overline{\ell^\infty(G,f)}$. 

Since $X$ admits no non-zero invariant Radon measure, $\ell^\infty_c(G,f) \cong C_c(X)$ admits no non-zero  invariant integrals. 
\end{proof}

\begin{example}[On the ideal $c_0(G)$] \label{ex:c_0}
For a countably infinite group $G$, the subspace $c_0(G)$ is a closed $G$-simple invariant ideal of $\ell^\infty(G)$, and the Pedersen ideal of $c_0(G)$ is $c_c(G)$. 
It follows that if $f$ is a positive non-zero function in $c_0(G)$, then
$$\ell^\infty_c(G,f) = c_c(G), \qquad \overline{\ell^\infty(G,f)} = c_0(G),$$
while $\ell^\infty(G,f)$ is some invariant hereditary ideal between these two ideals. Being $G$-simple, the ideal $c_0(G)$ is $G$-compact.

The only lower semi-continuous invariant integrals defined on $c_c(G)$ are multiples of the counting measure on $G$. Hence, if $f$ is a positive function in $c_0(G)$, then there is an invariant lower semi-continuous integral on $G$ normalized on $f$ if and only if $f$ belongs to $\ell^1(G)$. If a positive function $f$ in $c_0(G) \setminus \ell^1(G)$ is normalized by an invariant integral on $G$, which is the case if $G$ has the fixed-point property for cones, then this integral is necessarily singular. 

Consider now a countably infinite group $G$ with the fixed-point property for cones, and take a non-zero positive function $f \in c_0(G)$. The crossed product \Cs{} $c_0(G) \rtimes G$ is isomorphic to the \Cs{} $\cK$ of compact operators (on a separable infinite dimensional Hilbert space). There is an invariant integral on $G$ normalized for $f$, and we can view this integral as a invariant densely defined trace $\tau$ on $c_0(G)$ with $\tau(f)=1$. The image of $f$ in $c_0(G) \rtimes G$ is a positive compact operator with eigenvalues $\{f(t)\}_{t \in G}$. It was shown in {\cite{AGPS:traces}} that for a positive compact operator $T$ with eigenvalues $\{\lambda_n\}_{n=1}^\infty$, listed in decreasing order (and with multiplicity), there exists a densely defined trace on $\cK$ normalized for $T$ if and only if 
\begin{equation} \label{eq:AGPS}
\liminf_{n\to\infty} \sigma_{2n}/\sigma_n = 1,
\end{equation} 
where $\sigma_n = \sum_{j=1}^n \lambda_j$. 

Let $s > 0$ and choose $f \in c_0(G)$ such that the eigenvalues of the positive compact operator $f \in c_0(G) \rtimes G$ is the sequence $\{n^{-s}\}_{n=1}^\infty$. Then $f$ is trace class if and only if $s > 1$; and the limit in \eqref{eq:AGPS} is $1$ if and only if $s \ge 1$. If $s=1$, then $f$ is normalized by the \emph{Dixmier trace} on $\cK$. If $0 < s < 1$, then there is no trace on $c_0(G) \rtimes G$ (lower semi-continuous or not) that normalizes $f$. For such a choice of $f$, the invariant densely defined trace $\tau$ on $c_0(G)$ does not extend to a trace on $c_0(G) \rtimes G$, thus showing that the conclusion of  Lemma~\ref{prop:extending} fails without the assumption that the trace is lower semi-continuous.
\end{example}

\noindent By the remark in the example above, that for no positive functions $f$ in $c_0(G) \setminus \ell^1(G)$ does there exist an invariant lower semi-continuous integral $\mu$ on $\ell^\infty(G,f)$ with $\mu(f)=1$, we get the corollary below, which implies that at least some integrals witnessing the fixed-point property for cones for infinite groups  necessarily must be singular. 

\begin{corollary} \label{cor:finitegps} Let $G$ be a countable group with the property that for each positive function $f$ in $\ell^\infty(G)$ there exists a $G$-invariant \emph{lower semi-continuous} integral $\mu$ on $\ell^\infty(G,f)$ normalized for $f$. Then $G$ must be a finite group.
\end{corollary}

\section{The Roe algebra}
\label{sec:roe}

\noindent As an application of the results developed in the previous sections we shall in this last section prove some results about existence (and non-existence) of traces on the Roe algebra $\ell^\infty(G,\cK) \rtimes G$ associated with a (countably infinite) group $G$, where $\cK$ denotes the \Cs{} of compact operators on a separable infinite dimensional Hilbert space. These \Cs s originates from the thesis of John Roe, published in \cite{Roe:thesis}, where the index of elliptic operators are computed using traces on a (variant of) what is now called the Roe algebra.

The \Cs{} $\ell^\infty(G,\cK)$  is equipped with the natural action of $G$ given by left-translation. We shall consider existence (and non-existence) of invariant traces on this \Cs. 
 First we give a complete description of the densely defined lower semi-continuous traces on these \Cs s (there are not so many). The group $G$ plays no role in Lemma~\ref{lm:Ped(G,K)} and Proposition~\ref{prop:lsc-Roe} other than as a set, and it could be replaced with the set of natural numbers $\N$.

\begin{lemma} \label{lm:Ped(G,K)} The Pedersen ideal of $\ell^\infty(G,\cK)$ is equal to $\ell^\infty(G,\mathcal{F})$, where $\mathcal{F} = \cP(\cK)$ is the set of finite rank operators. 
\end{lemma}

\begin{proof} The ``point evaluation'' at $t \in G$ gives a surjective $^*$-homomorphism $\ell^\infty(G,\cK) \to \cK$ that maps $\cP(\ell^\infty(G,\cK))$ onto $\cP(\cK) = \mathcal{F}$. This shows that $\cP(\ell^\infty(G,\cK)) \subseteq \ell^\infty(G,\mathcal{F})$. Conversely, if $x \in \ell^\infty(G,\mathcal{F})$, then $x(t)$ has finite rank, for each $t \in G$, and so there exists a finite dimensional projection $p(t) \in \cK$ with $p(t)x(t) = x(t)$. The function $t \mapsto p(t)$ defines a projection $p$ in $\ell^\infty(G,\cK)$ satisfying $px=x$. As $p$ belongs to the Pedersen ideal, being a projection, it follows that also $x$ belongs to $\cP(\ell^\infty(G,\cK))$.
\end{proof}

\noindent It is worth mentioning that the Pedersen ideal of a \Cs{} of the form $C(T,\cK)$, where $T$ is a compact Hausdorff space, may be properly contained in $C(T,\cF)$, cf.\ \cite{GilTay:Pedersen-ideal}. 

For each $s \in G$, let $\tau_s$ be the (lower semi-continuous densely defined) trace on $\ell^\infty(G,\cK)$ given by $\tau_s(f) = \mathrm{Tr}(f(s))$, for $f$ either in $\ell^\infty(G,\cK)^+$ or in $\ell^\infty(G,\cF)$, where $\mathrm{Tr}$ is the standard trace on the compact operators $\cK$. 

In the proof of the proposition below we shall view the \Cs{} $\ell^\infty(G,\cK)$ as an $\ell^\infty(G)$-algebra via the natural (unital) embedding of $\ell^\infty(G)$ into the center of the multiplier algebra of $\ell^\infty(G,\cK)$.

\begin{proposition} \label{prop:lsc-Roe} For each countable group $G$, the cone, $T_{\mathrm{lsc}}(\ell^\infty(G,\cK))$, of densely defined lower semi-continuous traces on $\ell^\infty(G,\cK)$ is equal to the cone of finite positive linear combinations of the traces $\tau_s$, $s \in G$, defined above, i.e., traces of the form $\sum_{s \in G} c_s \tau_s$, where  $c_s \ge 0$ and $c_s \ne 0$ for only finitely many $s \in G$. 
\end{proposition}

\begin{proof} Clearly, any trace of the form $\sum_{s \in G} c_s \tau_s$, with $c_s \ge 0$ and $c_s \ne 0$ only for finitely many $s \in G$, is lower semi-continuous and densely defined. 

For the converse direction, take $\tau$ in $T_{\mathrm{lsc}}(\ell^\infty(G,\cK))$. Fix a one-dimensional projection $e \in \cK$. For each $s\in G$, let $e_s \in \ell^\infty(G,\cF)$ be given by $e_s(s)=e$ and $e_s(t) = 0$, when $t \ne s$. Set $c_s=\tau(e_s) \ge 0$.

We show first that the set $F = \{s \in G : c_s \ne 0\}$ is finite. Suppose it were infinite, and let $\{s_1,s_2,s_3, \dots\}$ be an enumeration of the elements in $F$. Let $p \in \ell^\infty(G,\cK)$ be a projection such that $\mathrm{Tr}(p(s_n)) \ge n c_{s_n}^{-1}$, for all $n \ge 1$. As $p \ge p \cdot 1_{\{s\}}$, for all $s \in G$, we get $\tau(p) \ge \tau(p \cdot 1_{\{s_n\}}) \ge c_{s_n} \mathrm{Tr}(p(s_n)) \ge n$. As this cannot be true for all $n\ge 1$, we conclude that $F$ is finite. 

Set $\tau_0 = \sum_{s \in F} c_s \tau_s$, and observe that $\tau_0(f) = \tau(f \cdot 1_{F})$, for all $f \in \ell^\infty(G,\cK)$. It follows that $\tau' := \tau-\tau_0$ is a positive trace on $\ell^\infty(G,\cK)$, satisfying $\tau'(f) = \tau(f \cdot 1_{F^c})$, for all $f \in \ell^\infty(G,\cK)$. We show that $\tau' = 0$. Assume, to reach a contradiction, that $\tau' \ne 0$. Notice that $\tau'$ vanishes on $c_c(G,\cK)$ by construction of $\tau_0$.  Since  the Pedersen ideal of $\ell^\infty(G,\cK)$ is generated (as a hereditary ideal) by its projections, cf.\ the proof of Lemma~\ref{lm:Ped(G,K)},  there is a projection $p \in \ell^\infty(G,\cK)$ such that $\tau'(p) > 0$. Write $G = \bigcup_{n=1}^\infty F_n$, where $\{F_n\}_{n \ge 1}$ is a strictly increasing sequence of finite subsets of $G$. Find a sequence $\{p_k\}_{k \ge 1}$ of pairwise orthogonal projections in $\ell^\infty(G,\cK)$ such that each $p_k$ is equivalent to $p$. 
Let $q \in \ell^\infty(G,\cK)$ be the projection given by
$$q(s) = p_1(s) + p_2(s) + \cdots + p_n(s), \qquad s \in F_{n+1} \setminus F_n,$$
for $n \ge 0$ (with the convention $F_0 = \emptyset$). Then 
$$q \cdot 1_{F_n^c} \ge (p_1+p_2+ \cdots + p_n) \cdot 1_{F_n^c},$$
for all $n \ge 1$; and as $\tau'(g \cdot 1_{E^c}) = \tau'(g)$, for all $g \in \ell^\infty(G,\cK)$ and all finite subsets $E \subseteq G$, we conclude that
$$\tau'(q) = \tau'(q \cdot 1_{F_n^c}) \ge \tau'((p_1+p_2+ \cdots + p_n) \cdot 1_{F_n^c}) = \tau'(p_1+p_2+ \cdots + p_n) = n \tau'(p),$$
for all $n\ge 1$, which is impossible. 
\end{proof}

\noindent A non-zero  trace of the form as in Proposition~\ref{prop:lsc-Roe} can clearly not be $G$-invariant when $G$ is infinite, so we obtain the following:

\begin{corollary} \label{cor:noinvtraces}
The \Cs{} $\ell^\infty(G,\cK)$ admits no non-zero lower semi-continuous invariant densely defined traces whenever $G$ is a countably infinite group, and, consequently, the Roe algebra $\ell^\infty(G,\cK) \rtimes G$ admits no non-zero densely defined lower semi-continuous trace.
\end{corollary}

\noindent When combining the conclusion of the corollary above with Theorem~\ref{thm:a}, we see that  the action of  $G$ on $T_{\mathrm{lsc}}(\ell^\infty(G,\cK))$ either must fail to be of cobounded type, or there is no non-zero locally bounded trace in $T_{\mathrm{lsc}}(\ell^\infty(G,\cK))$. If fact, both fail! That the latter fails follows easily from the description of $T_{\mathrm{lsc}}(\ell^\infty(G,\cK))$ in Proposition~\ref{prop:lsc-Roe}.

\begin{lemma} \label{lm:not-cobd}
If $G$ is a countably infinite group, then the action of $G$ on $T_\mathrm{lsc}(\ell^\infty(G,\cK))$ is not of cobounded type.
\end{lemma}

\begin{proof} Let $e$ be a positive contraction in  $\cP(\ell^\infty(G,\cK))$. We must find another positive contraction $a$ in $\cP(\ell^\infty(G,\cK))$ that is not tracially $G$-dominated by $e$, cf.\ Lemma~\ref{lm:cobounded0}. In other words,  for all finite sets $t_1,t_2, \dots, t_n \in G$, the inequality $\tau(a) \le \sum_{j=1}^n \tau(t_j.e)$ will fail for at least one $\tau \in T_{\mathrm{lsc}}(\ell^\infty(G,\cK))$, or, taking Proposition~\ref{prop:lsc-Roe} into account, 
$${\mathrm{Tr}}(a(s)) \le \sum_{j=1}^n {\mathrm{Tr}}(e(t_j^{-1}s)),$$
will fail for at least one $s \in G$.  It follows from Lemma~\ref{lm:Ped(G,K)} and its proof that $e$ is dominated by a projection in $\ell^\infty(G,\cK)$, so we may without loss of generality assume that $e$ itself is a projection. Set $f(s) = \mathrm{Tr}(e(s))$, for $s \in G$.

Let $G = \{s_1,s_2,s_3, \dots\}$ be an enumeration of the elements in $G$, and let $\{u_j\}_{j=1}^\infty$ be a sequence in which each element of $G$ is repeated infinitely often. Set $f_N = \sum_{j=1}^N u_j.f$, for each $N \ge 1$, and  let $g \colon G \to \N_0$ be given by $g(s_N) = f_N(s_N)+1$, for $N \ge 1$. For each finite set $t_1,t_2, \dots, t_n \in G$ there exists $N \ge 1$ such that $\sum_{j=1}^n t_j.f \le f_N$, which entails that $g(s_N) \nleq \sum_{j=1}^n t_j.f(s_N)$, for each $N \ge 1$.

Let now $a \in \ell^\infty(G,\cK)$ be a projection such that $\mathrm{Tr}(a(t)) = g(t)$, for all $t \in G$. Then $a$ is not tracially $G$-dominated by $e$.
\end{proof}

\noindent Although there are no \emph{densely defined} lower semi-continuous invariant traces on $\ell^\infty(G,\cK)$, when $G$ is infinite, there are still interesting  lower semi-continuous  traces on the Roe algebra. In the results to follow we describe the class of projections $p \in \ell^\infty(G,\cK)$ for which there exists a lower semi-continuous trace on $\ell^\infty(G,\cK)$ which is bounded and non-zero on the orbit of $p$, respectively, which is invariant and normalizes $p$. These two classes of projections agree when $G$ has the fixed-point property for cones, cf.\ Corollary~\ref{cor:b}. 
First we note the following general result that holds for locally finite groups:

\begin{proposition} \label{prop:locfinite}
The Roe algebra $\ell^\infty(G,\cK) \rtimes G$ of any locally finite group $G$ is stably finite. Moreover, for each non-zero projection $p \in \ell^\infty(G,\cK)$ there exists a lower semi-continuous trace $\tau \in T^+(\ell^\infty(G,\cK) \rtimes G)$ with $\tau(p)=1$, and hence 
an invariant lower semi-continuous trace  on $\ell^\infty(G,\cK)$ normalizing $p$.
\end{proposition}

\begin{proof} Write $G = \bigcup_{n=1}^\infty G_n$ as an increasing union of finite groups. The \Cs{} $\ell^\infty(G,\cK)$ is stably finite, as can be witnessed by the separating family of densely defined traces from Proposition~\ref{prop:lsc-Roe}. The crossed product $\ell^\infty(G,\cK) \rtimes G_n$ is stably finite because it embeds into the stably finite \Cs{} $\ell^\infty(G,\cK) \otimes B(\ell^2(G_n))$. It follows that
$$\ell^\infty(G,\cK) \rtimes G = \varinjlim \, \ell^\infty(G,\cK) \rtimes G_n$$
is stably finite, being an inductive limit of stably finite \Cs s.

Let next $p \in  \ell^\infty(G,\cK)$  be a non-zero projection. Then
\begin{equation} \label{eq:G_n}
p(\ell^\infty(G,\cK) \rtimes G)p = \varinjlim \, p(\ell^\infty(G,\cK) \rtimes G_n)p.
\end{equation}
The unital \Cs{} $p(\ell^\infty(G,\cK) \rtimes G_n)p$ admits a tracial state, for each  $n \ge 1$. To see this, choose $m \ge n$ such that the restriction $p'$ of $p$ to $\ell^\infty(G_m,\cK)$ is non-zero. The restriction mapping $\ell^\infty(G,\cK) \to \ell^\infty(G_m,\cK)$ is $G_n$-equivariant and therefore extends to a \sh{} 
$\ell^\infty(G,\cK) \rtimes G_n \to \ell^\infty(G_m,\cK) \rtimes G_n$, and in turns to a unital 
\sh{} 
$p(\ell^\infty(G,\cK) \rtimes G_n)p \to p'(\ell^\infty(G_m,\cK) \rtimes G_n)p'$. Composing this \sh{} with any tracial state on the (finite dimensional) \Cs{} $p'(\ell^\infty(G_m,\cK) \rtimes G_n)p'$ gives a tracial state on $p(\ell^\infty(G,\cK) \rtimes G_n)p$. As the inductive limit of a sequence of unital \Cs s with unital connecting mappings (e.g., as in \eqref{eq:G_n}) admits a tracial state if (and only if) each \Cs{} in the sequence does, we conclude that $p(\ell^\infty(G,\cK) \rtimes G)p$ admits a tracial state. By Remark~\ref{rem:unital} there is a lower semi-continuous trace on  $\ell^\infty(G,\cK) \rtimes G$ normalizing the projection $p$. The restriction of this trace to $\ell^\infty(G,\cK)$ is an invariant lower semi-continuous trace still normalizing $p$. 
\end{proof}

\noindent The Roe algebra of any infinite, locally finite group provides yet another example of a stably finite \Cs{} with an approximate unit consisting of projections which has no densely defined lower semi-continuous trace, 
cf.\ Theorem~\ref{thm:existtrace} and Proposition~\ref{prop:notrace}.
 
For the more general class of groups $G$ with the fixed-point property for cones, it follows from Corollary~\ref{cor:b} that if $p$ is a projection in $\ell^\infty(G,\cK)$, then there is a lower semi-continuous invariant trace on $\ell^\infty(G,\cK)$ that normalizes $p$ (and hence there is a trace on the Roe algebra $\ell^\infty(G,\cK) \rtimes G$ normalizing $p$),  if and only if there is a trace on $\ell^\infty(G,\cK)$, which is non-zero and bounded on the orbit $\{t.p\}_{t \in G}$. If $p$ has \emph{uniformly bounded dimension}, i.e., if $\sup_{t \in G} {\mathrm{Tr}}(p(t)) < \infty$, then such a trace  clearly exists, take for example  $\tau_s$ (defined above Proposition~\ref{prop:lsc-Roe}), for any fixed $s \in G$. One can do a bit better: In Proposition~\ref{cor:Roe-not-prop-inf} below it is shown that for each projection in $\ell^\infty(G,\cK)$
with uniformly bounded dimension there is a trace on the Roe algebra normalizing this projection provided that the group $G$ is supramenable (a formally weaker condition than having the fixed-point property for cones).

Let $G$ be a countably infinite group, and let $\ell \colon G \to \N_0$ be a proper length function on $G$, i.e., $\ell(t) = 0$ if and only if $t=e$, $\ell(s+t) \le \ell(s) + \ell(t)$, for all $s,t \in G$, and $W_n:= \{t \in G : \ell(t) \le n\}$ is finite, for all $n \ge 0$. Such proper length functions always exist, and if $G$ is finitely generated, then we can take $\ell$ to be the word length function with respect to some finite generating set for $G$. Set
\begin{equation} \label{eq:alpha-Z}
\alpha_n = \max_{t \in W_n} \mathrm{Tr}(p(t)), \qquad Z_n = \{t \in W_n : \mathrm{Tr}(p(t)) = \alpha_n\}.
\end{equation}
Let $\tau_n$ be the (lower semi-continuous densely defined) trace on $\ell^\infty(G,\cK)$ given by
$$\tau_n(f) = \frac{1}{|Z_n|} \sum_{t \in Z_n} \alpha_n^{-1} \,  \mathrm{Tr}(f(t)),$$
for $f$ in $\ell^\infty(G,\cK)^+$  or in $\ell^\infty(G,\cF)$.
Observe that $\tau_n(p) = 1$, for all $n \ge 1$. Let $\omega$ be a free ultrafilter on $\N_0$, and define a trace $\tau_{\omega,p}$ on the positive cone $\ell^\infty(G,\cK)^+$ by
$$\tau_{\omega,p}(f) = \lim_\omega \tau_n(f), \qquad f \in \ell^\infty(G,\cK)^+,$$
(where the limit along the ultrafilter is taken in the compact set $[0,\infty]$). Let $C_{\omega,p}$ be the cone of positive functions $f$ in $\ell^\infty(G,\cK)$, for which 
$\tau_{\omega,p}(f) < \infty$, and let $\cI_{\omega,p}$ be the linear span of $C_{\omega,p}$. Then $\cI_{\omega,p}$ is a hereditary symmetric ideal in $\ell^\infty(G,\cK)$, and $\tau_{\omega,p}$ defines a linear trace on $\cI_{\omega,p}$, cf.\ Proposition~\ref{prop:tau->tau'}. 

We say that the projection $p \in  \ell^\infty(G,\cK)$ has \emph{subexponentially growing dimension} if
\begin{equation} \label{eq:expgrowing}
\liminf_{n \to \infty} \frac{\alpha_{n+m}}{\alpha_n} = 1,
\end{equation}
for all $m \ge 0$. (This definition may depend on the choice of proper length function $\ell$, and should be understood to be with respect to \emph{some} proper length function.)

\begin{lemma} \label{lm:filter}
Let $\{\alpha_n\}_{n=0}^\infty$ be an increasing sequence of strictly positive real numbers satisfying \eqref{eq:expgrowing}. Then there is a free ultrafilter $\omega$ on $\N_0$ such that
$$\lim_\omega \frac{\alpha_{n+m}}{\alpha_n} = 1,$$
for all $m \ge 1$. 
\end{lemma}

\begin{proof} For each $m \ge 0$ and $\ep >0$, set $A_{m,\ep} = \{n \ge 0 : \alpha_{n+m}/\alpha_n \le 1+\ep\}$. By the assumption that \eqref{eq:expgrowing} holds,  each of the sets $A_{m,\ep}$ is infinite. The collection of sets $A_{m,\ep}$ is downwards directed, since $A_{m_1,\ep_1} \subseteq A_{m_2,\ep_2}$, when $m_1 \ge m_2$ and $\ep_1 \le \ep_2$. It follows that the intersection of any finite collection of these sets is infinite. We can therefore find a free ultrafilter $\omega$ which contains all the sets $A_{m,\ep}$; and any such ultrafilter will satisfy the conclusion of the lemma.
\end{proof}

\begin{lemma} \label{lm:subexpgrowth}
Let $G$ be a countably infinite group, let $p \in \ell^\infty(G,\cK)$ be a projection of subexponentially growing dimension, and let $\omega$ be a free ultrafilter as in Lemma~\ref{lm:filter} for the sequence $\{\alpha_n\}_{n=0}^\infty$ associated with the projection $p$ as in \eqref{eq:alpha-Z}. 
Then $\tau_{\omega,p}(p) = 1$ and  $\tau_{\omega,p}(t.p) \le 1$, for all $t \in G$.
\end{lemma}

\begin{proof} We have already noted that $\tau_n(p)=1$, for all $n \ge 1$, which implies that $\tau_{\omega,p}(p) = 1$. Let $m \ge 0$ and let $s \in W_m$. For $n \ge 0$ and $t \in Z_n \subseteq W_n$, we have $st \in W_{n+m}$, so that $\mathrm{Tr}(p(st)) \le \alpha_{n+m}$, which shows that
$$\tau_n(s^{-1}.p) = \frac{1}{|Z_n|} \sum_{t \in Z_n} \alpha_n^{-1} \, \mathrm{Tr}(p(st)) \le 
\frac{\alpha_{n+m}}{\alpha_n}.$$
By the assumption that $p$ has subexponentially growing dimension, by Lemma~\ref{lm:filter}, and by the choice of $\omega$, we conclude that $\tau_{\omega,p}(s^{-1}.p) \le 1$.
\end{proof}

\begin{theorem} \label{thm:subexpgrowth}
Let $G$ be a countably infinite group with the fixed-point property for cones and  let $p \in \ell^\infty(G,\cK)$ be a projection of subexponentially growing dimension. Then there is an invariant lower semi-continuous trace  on $\ell^\infty(G,\cK)$ normalized on $p$, and there is a lower semi-continuous trace on the Roe algebra $\ell^\infty(G,\cK) \rtimes G$ also normalized on $p$. 
\end{theorem}

\begin{proof} This follows from Corollary~\ref{cor:b}, where condition (i) is satisfied with $\tau = \tau_{p,\omega}$, cf.\ Lemma~\ref{lm:subexpgrowth}.
\end{proof}

\noindent We proceed to examine the case of projections in $\ell^\infty(G,\cK)$ of uniformly bounded dimension. In the lemma below we embed $\ell^\infty(G)$ into $\ell^\infty(G) \rtimes G$, which again embeds into (the upper left corner of) $(\ell^\infty(G) \rtimes G) \otimes M_n$, for each $n \ge 1$. Note  that projections in $\ell^\infty(G)$ are indicator functions $1_E$, for some subset $E$ of $G$. 

\begin{lemma} \label{lm:twoprojections}
If $G$ is an exact group and $n \ge 1$ is an integer.  Then for each projection $p$ in $(\ell^\infty(G) \rtimes G) \otimes M_n$ there is a projection $r$ in $\ell^\infty(G)$ such that $p$ and $r$ generate the same closed two-sided ideal in $(\ell^\infty(G) \rtimes G) \otimes M_n$.
\end{lemma}

\begin{proof} Let $n \ge 1$, let $p \in (\ell^\infty(G) \rtimes G) \otimes M_n$ be a projection, and let $\cI$ be the closed two-sided ideal in $(\ell^\infty(G) \rtimes G) \otimes M_n$ generated by $p$. It follows from \cite[Theorem 1.16]{Sie:ideals} that $\cI$ is the closed two-sided ideal in $(\ell^\infty(G) \rtimes G) \otimes M_n$ generated by $\mathcal{J}:=\cI \cap (\ell^\infty(G) \otimes M_n)$. Arguing as in the proof of \cite[Proposition 5.3]{KelMonRor:supra} we find a projection $q \in \mathcal{J}$ which generates the ideal $\cI$ in $(\ell^\infty(G) \rtimes G) \otimes M_n$. By \cite{Zhang:diagonal}, since $\ell^\infty(G)$ is of real rank zero, $q$ is equivalent to a diagonal projection $\mathrm{diag}(q_1,q_2, \dots, q_n)$ in $\ell^\infty(G) \otimes M_n$. Let $r \in \ell^\infty(G)$ be the supremum of the projections $q_1, q_2, \dots, q_n$. Then $q$ and $r$ generate the same ideal of $\ell^\infty(G) \otimes M_n$, and so $r$ and $p$ generate the same ideal in $(\ell^\infty(G) \rtimes G) \otimes M_n$.
\end{proof}

\noindent The result below extends the characterization in \cite{KelMonRor:supra} of supramenable groups in terms of non-existence of properly infinite projections in the uniform Roe algebra.

\begin{proposition} \label{cor:Roe-not-prop-inf} Let $G$ be a group. The following conditions are equivalent:
\begin{enumerate}
\item $G$ is supramenable.
\item Each non-zero projection in the stabilized uniform Roe algebra $(\ell^\infty(G) \rtimes G) \otimes \cK$ is normalized by a lower semi-continuous trace in $T^+((\ell^\infty(G) \rtimes G) \otimes \cK)$.
\item Each projection in $\ell^\infty(G,\cK)$ of uniformly bounded dimension is normalized by  an invariant  lower semi-continuous trace in $T^+(\ell^\infty(G,\cK))$ and by a lower semi-continuous trace in $T^+(\ell^\infty(G,\cK) \rtimes G)$
\item The stabilized uniform Roe algebra $(\ell^\infty(G) \rtimes G) \otimes \cK$ contains no properly infinite projections.
\end{enumerate}
\end{proposition}

\begin{proof} (i) $\Rightarrow$ (ii).  Since $(\ell^\infty(G) \rtimes G) \otimes \cK$ is the inductive limit of \Cs s of the form $(\ell^\infty(G) \rtimes G) \otimes M_n$, for $n \ge 1$, it suffices to show that each projection $p$ in $(\ell^\infty(G) \rtimes G) \otimes M_n$ is normalized by a lower semi-continuous trace on this \Cs.

By Lemma~\ref{lm:twoprojections} there is a projection $q = 1_E \in \ell^\infty(G)$ that generates the same closed two-sided ideal in  $(\ell^\infty(G) \rtimes G) \otimes M_n$ as $p$. Since $G$ is supramenable, the set $E$ must be non-paradoxical, so by Tarski's theorem there is an invariant trace $\tau_0$ on $\ell^\infty(G,q)$ which normalizes $q$ (see \cite[Proposition 5.3]{KelMonRor:supra}). Extend $\tau_0$ to a trace $\tau$ on $\cJ:=\cJ_{(\ell^\infty(G) \rtimes G) \otimes M_n}(q)$ satisfying $\tau(q)=1$. Now, $p \in \cJ$ and $0 < \tau(p) < \infty$, by the assumption on $p$ and $q$, so upon rescaling we obtain a trace $\tau$  on $\cJ$ satisfying $\tau(p)=1$; and $\tau$ is lower semi-continuous by the remarks below Theorem~\ref{thm:GKP}. Finally, we can extend $\tau$ to a trace in $T^+((\ell^\infty(G) \rtimes G) \otimes M_n)$ using \eqref{eq:tau}.

(ii) $\Rightarrow$ (iii). Let $p$ be a projection in $\ell^\infty(G,\cK)$ such that $n:= \sup_{t \in G} \mathrm{Tr}(p(t))< \infty$. Let $e \in \cK$ be a projection of dimension $n$, and let $\overline{e} \in \ell^\infty(G,\cK)$ be the projection given by $\overline{e}(t) = e$, for all $t \in G$. Then $\overline{e}$ is fixed by the action of $G$, and we have isomorphisms 
$$\overline{e}(\ell^\infty(G,\cK))\overline{e} \cong \ell^\infty(G) \otimes M_n, \qquad 
\overline{e}(\ell^\infty(G,\cK) \rtimes G)\overline{e} \cong (\ell^\infty(G)\rtimes G) \otimes M_n.$$
Moreover, $p$ is equivalent to a projection $p_0$ in $\overline{e}(\ell^\infty(G,\cK))\overline{e}$, which, under the isomorphism above, corresponds to a projection $p_1 \in \ell^\infty(G) \otimes M_n$. By (ii) there is a lower semi-continuous trace on $(\ell^\infty(G)\rtimes G) \otimes M_n$ normalizing $p_1$. Hence there is a lower semi-continuous trace on $\overline{e}(\ell^\infty(G,\cK) \rtimes G)\overline{e}$ normalizing $p_0$. Arguing as in Remark~\ref{rem:unital} we can extend this trace to a lower semi-continuous trace  in  $T^+(\ell^\infty(G,\cK) \rtimes G)$. The restriction of $\tau$  to $\ell^\infty(G,\cK)$  becomes an invariant lower semi-continuous  trace. 

(iii) $\Rightarrow$ (iv) is clear (no properly infinite projection can be normalized by a trace). If $G$ is non-supramenable, then $\ell^\infty(G) \rtimes G$ contains a propery infinite projection, cf.\  \cite[Proposition 5.5]{RorSie:action}, in which case (iv) cannot be true, which proves (iv) $\Rightarrow$ (i).
\end{proof}

\noindent It was shown in \cite{RorSie:action} that the uniform Roe algebra $\ell^\infty(G) \rtimes G$ is properly infinite, i.e., its unit is a properly infinite projection, if and only if $G$ is non-amenable; and in \cite{KelMonRor:supra} it was shown that the uniform Roe algebra \emph{contains} a properly infinite projection if and only if $G$ is non-supramenable. The proposition above allows us to conclude that no matrix algebras over the uniform Roe algebra contains a properly infinite projections when $G$ is supramenable. 

It was shown by Wei in \cite{Wei:qd} and Scarparo in \cite{Sca:locfinite}, answering a question in \cite{KelMonRor:supra}, that  $\ell^\infty(G) \rtimes G$  is \emph{finite} if and only if $G$ is a locally finite group. Using a similar idea as in  \cite{Sca:locfinite} we show below that the Roe algebra $\ell^\infty(G,\cK) \rtimes G$ contains a properly infinite projection whenever $G$ is not locally finite. We may therefore strengthen Proposition~\ref{prop:locfinite} as follows: The Roe algebra $\ell^\infty(G,\cK) \rtimes G$ is stably finite if \emph{and only if} $G$ is locally finite; and if $\ell^\infty(G,\cK) \rtimes G$ is not stably finite, then it contains a \emph{properly infinite} projection. 

\begin{lemma} \label{lm:fin-gen-inf} Let $G$ be a non-locally finite group. Then there is a finite subset $S$ of $G$ such that for  each integer $N \ge 1$ there exists a non-zero function $\varrho \colon G \to \N_0$ satisfying
\begin{equation} \label{eq:Nf}
N\varrho \le \sum_{s \in S} s.\varrho.
\end{equation}
\end{lemma}

\begin{proof}  Let $G_0$ be an infinite, finitely generated subgroup of $G$, and let $S$ be a finite symmetric generating set for $G_0$. 
By \cite[Lemma 1]{Zuk:isoperimetric} there is a one-sided geodesic $\{t_n\}_{n =0}^\infty$ in $G_0$, where $t_{n}t_{n+1}^{-1} \in S$, for all $n \ge 0$; and $n \mapsto t_n$ is injective. Fix $N \ge 1$ and define $\varrho \colon G\to \N_0$ by $\varrho(t_n) = N^n$, for all $n \ge 0$; and set $\varrho(t)=0$ if $t \notin \{t_n : n \ge 0\}$. 

Fix $n \ge 0$ and set $s = t_{n}t_{n+1}^{-1} \in S$. Then $s.\varrho(t_{n}) = \varrho(t_{n+1}) = N\varrho(t_n)$. As $\varrho(t) =0$, when $t \notin \{t_n : n \ge 0\}$, we see that \eqref{eq:Nf} holds.
\end{proof}

\begin{proposition} \label{prop:propinfproj}
For each countable non-locally finite group  $G$  there is a projection $p \in \ell^\infty(G,\cK)$ satisfying:
\begin{enumerate}
\item $p$ is properly infinite in the Roe algebra $\ell^\infty(G,\cK) \rtimes G$.
\item There is no trace on the Roe algebra $\ell^\infty(G,\cK) \rtimes G$ normalizing the projection $p$.
\item Each  trace on $\ell^\infty(G,\cK)$ is either unbounded or zero on the orbit $\{t.p\}_{t \in G}$. 
\end{enumerate}
\end{proposition}

\begin{proof} Note first that if $e$ and $f$ are projections in $\ell^\infty(G,\cK)$, then $e \sim f$, respectively, $e \precsim f$, in $\ell^\infty(G,\cK)$ if and only if $\mathrm{Tr}(e(t)) = \mathrm{Tr}(f(t))$, respectively, $\mathrm{Tr}(e(t)) \le \mathrm{Tr}(f(t))$, for all $t \in G$. 
Let $\varrho \colon G \to \N_0$ be as in Lemma~\ref{lm:fin-gen-inf} with respect to a finite subset $S$ of $G$ and with $N \ge 2|S|$. Find a projection $q \in \ell^\infty(G,\cK)$ with $\mathrm{Tr}(q(t)) = \varrho(t)$, for all $t \in G$. 

Let $p$ and $p'$ be the the $|S|$-fold, respectively, the $N$-fold, direct sum of $q$ with itself in $\ell^\infty(G,\cK)$, and set $e = \bigoplus_{s \in S} s.q$. Then $\mathrm{Tr}(p'(t)) = N\varrho(t)$ and $\mathrm{Tr}(e(t)) = \sum_{s\in S} s.\varrho(t)$, for all $t \in G$, so $p' \precsim e$  and $p \oplus p \precsim p'$ in $\ell^\infty(G,\cK)$. Moreover, $p \sim e$ in $\ell^\infty(G,\cK) \rtimes G$ (since $q \sim t.q$ in $\ell^\infty(G,\cK) \rtimes G$, for all $t \in G$). Hence,
$$p \oplus p \precsim p' \precsim e \sim p,$$
in $\ell^\infty(G,\cK) \rtimes G$, which shows that (i) and (ii) hold. 

(iii) follows from (ii) when $G$ has the fixed-point property for cones, cf.\ Corollary~\ref{cor:b}, but requires a separate argument for general groups. Applying the operator $\sigma \mapsto \sum_{s \in S} s.\sigma$, on functions $\sigma \colon G \to \N_0$, to the left and right hand side of \eqref{eq:Nf} $k-1$ times, we get that $N^k \varrho \le \sum_{s \in S^k} \overline{s}.\varrho$, for all $k \ge 1$, where $S^k$ is the set of $k$-tuples of elements from $S$, and $\overline{s} \in G$ is the product of the $k$ elements in the $k$-tuple $s \in S^k$. For $k \ge 1$, set $e_k = \bigoplus_{s \in S^k} \overline{s}.q$, and let $p'_k$ be the $N^k$-fold direct sum of $q$ with itself. Then $\mathrm{Tr}(p'_k(t)) = N^k \varrho(t)$ and $\mathrm{Tr}(e_k(t)) = \sum_{s \in S^k} \overline{s}.\varrho(t)$, for all $t \in G$, so $p'_k \precsim e_k$.

Let $\tau$ be any trace in $T^+(\ell^\infty(G,\cK))$. Then 
$$N^k \tau(q) = \tau(p'_k) \le \tau(e_k) = \sum_{s \in S^k} \tau(\overline{s}.q) \le |S|^k \sup_{t \in G} \tau(t.q).$$
As this holds for all $k \ge 1$, either $\tau(q)=0$ or $\sup_{t \in G} \tau(t.q) = \infty$. Suppose that $\{\tau(t.p)\}_{t \in G}$ is bounded. Then  $\{\tau(t.q)\}_{t \in G}$ is also bounded. Fix $t \in G$, and let $\tau'$ be the trace on $\ell^\infty(G,\cK)$ given by $\tau'(f) = \tau(t.f)$, for $f \in \ell^\infty(G,\cK)^+$. As $\tau'$ also is bounded on the orbit $\{s.q\}_{s \in G}$, the argument above implies that $0= \tau'(q) = \tau(t.q)$, so $\tau(t.p) = |S|\tau(t.q) = 0$. This proves that $\tau$ is zero on the orbit $\{t.p\}_{t \in G}$.
\end{proof}

\noindent Any projection satisfying  the conclusions of Proposition~\ref{prop:propinfproj} (iii) above must have exponentially growing dimension with respect to any proper length function on the group, cf.\ 
Lemma~\ref{lm:subexpgrowth}. 

{\small{
\bibliographystyle{amsplain}
\providecommand{\bysame}{\leavevmode\hbox to3em{\hrulefill}\thinspace}
\providecommand{\MR}{\relax\ifhmode\unskip\space\fi MR }
\providecommand{\MRhref}[2]{%
  \href{http://www.ams.org/mathscinet-getitem?mr=#1}{#2}
}
\providecommand{\href}[2]{#2}

}}

\vspace{1cm}

\noindent Mikael R\o rdam \\
Department of Mathematical Sciences\\
University of Copenhagen\\ 
Universitetsparken 5, DK-2100, Copenhagen \O\\
Denmark \\
rordam@math.ku.dk\\

\end{document}